\documentclass[12pt, a4paper, leqno]{amsart}

\usepackage[utf8]{inputenc}
\usepackage{amsmath} 
\usepackage{amsfonts}
\usepackage{amssymb}
\usepackage{txfonts}   
\usepackage{amsfonts}     
\usepackage{graphicx}    
\usepackage[dvipsnames]{xcolor}
\usepackage[colorlinks, allcolors=RedViolet]{hyperref}

\newcommand{\R}{\varmathbb{R}}

\newcommand{\Rn}{{\varmathbb{R}^n}}

\newcommand{\Ha}{\mathcal{H}}

\newcommand{\C}{\mathcal{C}}

\def\diam{\qopname\relax o{diam}}
\def\loc{\qopname\relax o{loc}}

\def\dist{\qopname\relax o{dist}}

\def\diam{\qopname\relax o{diam}}
\def\inte{\qopname\relax o{int}}

\usepackage{pgfplots}
\pgfplotsset{compat=1.15}
\usepackage{mathrsfs}
\usetikzlibrary{arrows}

\usepackage{amsthm}

\swapnumbers
\theoremstyle{plain}
\newtheorem{theorem}[equation]{Theorem}
\newtheorem{lemma}[equation]{Lemma}
\newtheorem{proposition}[equation]{Proposition}
\newtheorem{corollary}[equation]{Corollary}

\theoremstyle{definition}
\newtheorem{definition}[equation]{Definition}

\theoremstyle{remark}
\newtheorem{remark}[equation]{Remark}

\numberwithin{equation}{section}

\makeindex             

\title{On Choquet integrals and pointwise estimates}

\author{Petteri Harjulehto} 
\address[Petteri Harjulehto]{Department of Mathematics and Statistics,
FI-00014 University of Helsinki, Finland} \email{petteri.harjulehto@helsinki.fi}

\author{Ritva Hurri-Syrj\"anen}
\address[Ritva Hurri-Syrj\"anen]{Department of Mathematics and Statistics,
FI-00014 University of Helsinki, Finland} \email{ritva.hurri-syrjanen@helsinki.fi}

\date{\today}

\begin{document}

\begin{abstract}We consider inequalities where integrals are defined in the sense of  Choquet with respect to Hausdorff content.
We study cases where continuously differentiable functions
are defined on open, connected sets with  so much regularity that
there exists  a pointwise estimate  between 
the values of a function and its gradient under the maximal  operator or the Riesz potential,
at every point of the set.
We show that certain  Hardy inequalities and Poincare-Sobolev  inequalities are 
valid  in this context.
\end{abstract}


\maketitle

%%%%%%%%%%%%%%%%%%%%%%%%%%%%%%%%%%%%%%%%%%%%%%%%%%%
%%%%%%%%%%%%%%%%%%%%%%%%%%%%%%%%%%%%%%%%%%%%%%%%%%%%%%%
\section{Introduction}

We consider inequalities where integrals are defined in the sense of  Choquet with respect to the Hausdorff content
$\Ha_\infty^{\delta} $, $0<\delta \le n$.
Let us recall that for any non-negative  function  $u$ defined on
an open, connected set $\Omega$  in the Euclidean $n$-space, $n\geq 2$,   the integral in the sense of Choquet with respect to Hausdorff content is defined by
\begin{equation}
\int_\Omega u(x) \, d \Ha^{\delta}_\infty := \int_0^\infty \Ha^{\delta}_\infty\big(\{x \in \Omega : u(x)>t\}\big) \, dt. 
\end{equation}
The definition goes back to \cite{Cho53}.
Research of the properties of these integrals  in potential theory was initiated by  David R. Adams \cite{Adams1975}, \cite{Adams1986}, \cite{Adams1998},
and continued in \cite{Den94}, \cite{DX}, \cite{AX2012}.
Recently, the study has revived and  new results have been published
\cite{PonceSpector20}, \cite{MartinezSpector2021}, \cite{OoiPhuc22}, \cite{HH-S_JFA}, \cite{PonceSpector23}, \cite{HH-S_AAG},
\cite{ChenSpector2023}, \cite{ChenOoiSpector2023}.
In the present paper we are interested in results for smooth functions defined on sets with enough regularity.

The underlying assumption is that 
sets $\Omega$
have  so much regularity that
there exists  a pointwise estimate 
between 
the values of any continuously differentiable function and its gradient under the maximal  operator or the Riesz potential,
at every point of $\Omega$.
Then, we show that some
Hardy and  Poincar\'e-Sobolev  inequalities are valid for functions defined on $\Omega$.
Whenever $\delta =n$, the inequalities recover some  corresponding  known classical cases.
The boundedness of the fractional Hardy-Littlewood maximal operator \cite[Theorem 7(a)]{Adams1998} by Adams is essential in some of the proofs.
We recall the basic properties of Hausdorff content and Choquet integrals  in Section \ref{Preliminaries} in order to make the text as self contained as possible.
The Hardy inequalities 
for continuously differentiable functions with compact support 
are considered in  Section \ref{Section-Hardy}. 
The main theorem there implies  that the Hardy-type inequality is valid for functions defined on convex domains, but also  on some other domains which are not convex
as the  following result shows.
\begin{corollary}\label{HardyCorollary}
Let $\Omega$ be an annulus in $\Rn$, $n\geq 2$,
\begin{equation*}
\Omega =\{x\in\Rn : 0<r<\vert x\vert < R <\infty\}\,.
\end{equation*}
Let $\delta \in (0, n]$, $\kappa \in [0, 1)$,  $p\in (\delta /n, \delta/\kappa )$.
Then there exists a positive constant $c <\infty$  depending only on  $n$, $\delta$,  $\kappa$, $p$, $r$, and $R$ such that
\begin{equation}\label{HardyIneq}
\int_{\Omega}\frac{\vert u(x)\vert^p}{\dist (x,\partial \Omega)^{p(1-\kappa)}}\,d\Ha_{\infty}^{\delta-\kappa p}\le
c\int_{\Omega}\vert \nabla u(x)\vert^p\,d\Ha_{\infty}^{\delta}
\end{equation}
for all $u\in C^{1}_0(\Omega )$.
\end{corollary}
Whenever $\delta =n$ and $\kappa=0$, a well-known Hardy inequality result  is recovered
\cite{Lewis}, \cite{Wannebo}, \cite[Lemma 3.1]{EH-S}.
We take an agreement that $\delta /0 :=  \infty$.

Poincare-Sobolev  inequalities are 
studied in Section \ref{sJohn}.
One of the results in this section is that the Poincar\'e inequality
formulated  using  Choquet integrals with respect to Hausdorff content   is  valid in an $s$-spire, that is a power cusp.
Whenever $\delta =n$ and
$s\in[1,n/(n-1))$,
Corollary \ref{PoincareCorollary} recovers the classical result  \cite{Mazya}, \cite{MazyaPoborchi}.
\begin{corollary}\label{PoincareCorollary}
Suppose that
$\Omega$ is an $s$-spire in $\Rn$, $n\geq 2$,  
\begin{equation*}
\Omega =\{(x_1, x_2, \dots , x_n)\in (0,1)\times \R^{n-1}\, : \vert\vert (x_2, x_3, \dots ,x_n)\vert\vert < x^s\}\,,
\end{equation*}
where $s\in [1,n/(n-1))$.
Let $\delta\in (0,n]$ be given.
\begin{itemize}
\item[(a)]\, The Poincar\'e inequality.\\
If     $u \in C^1(\Omega)$ and     $p>\delta /n$, then
\begin{equation*}
\inf_{b \in \R} \int_\Omega |u(x)-b|^p \, d \Ha^{\delta}_\infty
\le c  \int_\Omega |\nabla u(x)|^p \, d \Ha^{\delta}_\infty\,,
\end{equation*}
where $c$ is a constant which depends only on $n$, $\delta$, $p$, and  $s$.\\
\item[(b)]\, A weak-type estimate.\\
If $u \in C^1(\Omega)$ and
$p=\delta /n$,  then 
for every $t>0$
\begin{equation*}
\inf_{b \in \R} \Big(\Ha_\infty^{\delta}\bigl(\{x\in\Omega :\vert u(x)-b\vert >t\}\bigr)\Big)^{\frac{n-1}{\delta}}
\le \frac{c}{t^{\frac{1}{s}}} \Big(\int_{\Omega}\vert \nabla u(x)\vert^{\frac{\delta}{n}}\,d\Ha_\infty^{\delta} \Big)^{\frac{n}{s\delta}}\,
\end{equation*}
where $c$ is 
a constant which depends only on $n$, $\delta$, $p$, and $s$.
\end{itemize}

\end{corollary}

%%%%%%%%%%%%%%%%%%%%%%%%%%%%%%%%%%%%%%%%%%%%%%%%%%%%%%%%%%%%%%%%%%%%%%%%%%%%%%%%%%%
\section{Hausdorff content and the Choquet integral}\label{Preliminaries}

We recall the definition of the Hausdorff content of a set $E$ in $\Rn$, \cite[2.10.1, p.~169]{Federer},  \cite{Adams1998}, \cite{Adams2015}.
An  open ball centred at $x$ with radius $r>0$ is written as $B(x,r)$.

\begin{definition}[Hausdorff content]\label{defn:Haus}
Let $E$ be a set in $\Rn$, $n \ge 2$. Suppose that
$\delta \in (0, n]$.
The $\delta$-dimensional Hausdorff content of $E$ is defined by
\begin{equation}
\Ha_\infty^{\delta} (E) := \inf \bigg\{ \sum_{i=1}^\infty r_i^{\delta}: E \subset \bigcup_{i=1}^\infty B(x_i, r_i)\bigg\},\label{HausdorffC}
\end{equation}
where the infimum is taken over all  
countable  (or finite ) ball coverings of $E$.
The quantity \eqref{HausdorffC} is also called  the $\delta$-Hausdorff capacity or briefly 
 the Hausdorff content.
\end{definition}
The Hausdorff content has the following properties:
\begin{enumerate}
\item[(H1)] $\Ha_\infty^{\delta}(\emptyset) =0$;
\item[(H2)] if $A \subset B$ then $\Ha_\infty^{\delta}(A) \le \Ha_\infty^{\delta} (B)$;
\item[(H3)]  if $E \subset \Rn$ then 
\[
\Ha_\infty^{\delta}(E) = \inf_{E \subset U \text{ and }U \text{ is open}}\Ha_\infty^{\delta}(U); 
\]
\item[(H4)] if $(K_i)$ is a decreasing sequence of compact sets then 
\[
\Ha_\infty^{\delta}\Big(\bigcap_{i=1}^\infty K_i \Big)
= \lim_{i \to \infty} \Ha_\infty^{\delta}(K_i);
\]
\item[(H5)] if $(A_i)$ is any sequence of sets then
\[
\Ha_\infty^{\delta}\Big(\bigcup_{i=1}^\infty A_i \Big)
\le  \sum_{i=1}^\infty \Ha_\infty^{\delta}(A_i).
\]
\end{enumerate}

Since the Hausdorff content
$\Ha_\infty^{\delta}$
 is  not a capacity in the sense of Choquet \cite{Cho53},  it is good to have a more refined dyadic counterpart of
 $\Ha_{\infty}^{\delta}$, that is
\begin{equation}
\tilde{\Ha}_\infty^{\delta} (E) := \inf \bigg\{ \sum_{i=1}^\infty \ell (Q_i)^{\delta}: E \subset \inte\big(\bigcup_{i=1}^\infty Q_i \big)\bigg\},\label{HausdorffCD}
\end{equation}
where the infimum is taken over all dyadic cube coverings, and  $E$ is a subset of the interior of  the union of cubes.  Here $\ell (Q)$ is the side length of a cube $Q$. This version has been introduced and studied by D.Yang and W. Yuan in \cite{YangYuan08}.
It is known that  $\Ha_\infty^{\delta}(E) $ and $\tilde{\Ha}_\infty^{\delta}(E)$ are comparable to each other  for all sets $E$ in $\Rn$,
that is there are finite, positive constants $c_1(n)$ and $c_2(n)$ such that
\begin{equation*}
c_1(n)\Ha_\infty^{\delta}(E)\le\tilde{\Ha}_\infty^{\delta}(E) \le c_2(n)\Ha_\infty^{\delta}(E)\,,
\end{equation*}
we refer to  \cite[Proposition 2.3]{YangYuan08}.
By \cite[Theorem 2.1]{YangYuan08} $\tilde{\Ha}_\infty^{\delta}$ is a Choquet capacity for all $0<\delta \le n$.
By \cite[Proposition 2.4]{YangYuan08}  we have
\[
\tilde{\Ha}_\infty^{\delta}(A_1 \cup A_2) + \tilde{\Ha}_\infty^{\delta}(A_1 \cup A_2) \le \tilde{\Ha}_\infty^{\delta}(A_1) + \tilde{\Ha}_\infty^{\delta}(A_2)
\] 
for all $A_1, A_2 \subset \Rn$. Hence, $\tilde{\Ha}_\infty^{\delta}$ is a strongly subadditive Choquet capacity for all $0<\delta \le n$. 

The missing  property of  $\Ha_\infty^{\delta}$ is valid for $\tilde{\Ha}_\infty^{\delta}$ : 
\newline
(H6)  If $(E_i)$ is an increasing sequence of  sets, then 
\begin{equation*}
\tilde{{\Ha}}_\infty^{\delta}\big(\bigcup_{i=1}^\infty E_i \big)
= \lim_{i \to \infty} \tilde{{\Ha}}_\infty^{\delta}(E_i).
\end{equation*}

Recently, properties of Hausdorff content have been studied extensively
\cite{DX},
\cite{YangYuan08}, \cite{Tang},  \cite{SaitoTanakaWatanabe2016}, \cite{Liu16}, 
\cite{SaitoTanakaWatanabe2019}, \cite{Saito2019}, \cite{SaitoTanaka2022}, \cite{PonceSpector23},
\cite{ChenSpector2023}.

Let $0 <\delta\le n$. The definition of the $\delta$-dimensional  Hausdorff measure for $E \subset \R^n$ is
\[
\Ha^\delta (E) := \lim_{\rho \to 0^+}  \inf \bigg\{ \sum_{i=1}^\infty r_i^{\delta}: E \subset \bigcup_{i=1}^\infty B(x_i, r_i) \text{ and } r_i \le \rho \text{ for all } i\bigg\},
\]
where the infimum is taken over 
 all such countable  (or finite) ball coverings of $E$ that the radius of a  ball is at most $\rho$.
 For all sets $E\subset\R^n$,
 \begin{equation*}
 {\Ha}_{\infty}^{n} (E)\le\Ha^{n} (E) \le
 c(n) {\Ha}_{\infty}^{n} (E),
 \end{equation*}
 \cite[Proposition 2.5]{HH-S_JFA}. Thus there are constants $c_1(n)$ and $c_2(n)$ such that
 \begin{equation*}
 c_1(n) {\Ha}_{\infty}^{n} (E)\le
 |E| \le
 c_2(n) \Ha^{n}_\infty (E),
 \end{equation*}
 for any  Lebesgue measurable  set $E \subset \Rn$.
Here $|E|$ is the Lebesgue measure of $E$.

 We recall the definition of the Choquet integral.
 In the present paper $\Omega$  is always assumed to be 
  an open, connected set 
 in $ \Rn$, $n \ge 2$,   unless stated otherwise, and  hence it  is called a domain.
 \begin{definition}[Choquet integral]
For any function $f:\Omega\to [0,\infty]$ the integral in the sense of Choquet with respect to Hausdorff content is defined by
\begin{equation}\label{IntegralDef}
\int_\Omega f(x) \, d \Ha^{\delta}_\infty := \int_0^\infty \Ha^{\delta}_\infty\big(\{x \in \Omega : f(x)>t\}\big) \, dt. 
\end{equation}
\end{definition}
The right-hand side is well defined as a Lebesgue integral.
For the properties of Choquet integrals we refer to
\cite{Adams1986}, 
\cite{Adams1998}, \cite{Den94}, \cite{OV}, \cite{DX},  \cite{AX2012}, \cite{Adams2015}, \cite{Kawabe19},
\cite{HH-S_AAG}, \cite{HH-S_JFA}, \cite{PonceSpector23}.

For the convenience of the reader we list here the basic properties of
the Choquet integral with respect to Hausdorff content.
\begin{enumerate}
\item[(C1)] $ \displaystyle \int_\Omega a f(x) \, d \Ha^{\delta}_\infty = a \int_\Omega  f(x) \, d \Ha^{\delta}_\infty$ for every $a\ge 0$;
\item[(C2)] $\displaystyle \int_\Omega f(x) \, d \Ha^{\delta}_\infty=0$ if and only if $f(x)=0$  for $\Ha^{\delta}_\infty$-almost every $x\in \Omega$;
\item[(C3)] $\displaystyle \int_\Omega \chi_E(x) \, d \Ha^{\delta}_\infty = \Ha^{\delta}_\infty(\Omega \cap E)$;
\item[(C4)] if $A\subset B$, then $\int_A f(x) \, d \Ha^{\delta}_\infty \le \int_B f(x) \, d \Ha^{\delta}_\infty$;
\item[(C5)] if $0\le f\le g$, then $\displaystyle \int_\Omega f(x) \, d \Ha^{\delta}_\infty \le \int_\Omega g(x) \, d \Ha^{\delta}_\infty$;
\item[(C6)] $\displaystyle \int_\Omega f(x)+g(x) \, d \Ha^{\delta}_\infty \le 2\Big(\int_\Omega f(x) \, d \Ha^{\delta}_\infty + \int_\Omega g(x) \, d \Ha^{\delta}_\infty\Big)$;
\item[(C7)] $\displaystyle \int_\Omega f(x)g(x) \, d \Ha^{\delta}_\infty \le 2\Big(\int_\Omega f(x)^p \, d \Ha^{\delta}_\infty\Big)^{1/p} \Big( \int_\Omega g(x)^q \, d \Ha^{\delta}_\infty\Big)^{1/q}$   when   $p, q>1$
are H\"older conjugates, that is $\frac{1}{p}+\frac{1}{q}=1$.
\end{enumerate}
For the proofs of these properties  we refer to \cite{Adams1998} and \cite[Chapter 4]{Adams2015} .

We note  that by changing of  the variables, $t^{1/p} = \lambda$ yields 
\[
\begin{split}
\int_0^\infty \Ha^{\delta}_\infty\big(\{x \in \Omega : f(x)^p>t\}\big) \, dt &=
\int_0^\infty \Ha^{\delta}_\infty\big(\{x \in \Omega : f(x)>t^{1/p}\}\big) \, dt\\
&= \int_0^\infty p \lambda^{p-1 }\Ha^{\delta}_\infty\big(\{x \in \Omega : f(x)>\lambda\}\big) \, d\lambda.
\end{split}
\]

\begin{remark}
Since $\Ha^{n}_\infty$ and Lebesgue measure are comparable for all measurable sets, there exist constants $c_1(n)$ and $c_2(n)$ such that
\begin{equation*}
c_1(n)  \int_\Omega |f(x)|  \, d \Ha^{ n}_\infty \le \int_\Omega |f(x)| \, dx \le c_2(n)  \int_\Omega |f(x)|  \, d \Ha^{ n}_\infty
\end{equation*}
\end{remark}

We recall a useful lemma from  \cite[Lemma 2.2]{ChenOoiSpector2023} and \cite[Proposition 2.3]{HH-S_La}.

\begin{lemma}\label{GeneralizationOV}
 Let $\Omega$ be an open subset of $\Rn$. If  $0 < \delta_1 \le n$ and $0 < \delta_2 \le n$ such that $\delta_1 <\delta_2$, then
 the  inequality
\begin{equation*}
\int_\Omega |f(x)|\Ha_{\infty}^{\delta_2} \, dx 
\le 
 \frac{\delta_2}{\delta_1} \Big(\int_\Omega |f(x)|^{\frac{\delta_1}{\delta_2}} \, d \Ha^{\delta_1}_\infty \Big)^{\frac{\delta_2}{\delta_1}}
\end{equation*}
holds for all measurable  functions $f: \Omega \to [-\infty, \infty]$.
\end{lemma}

For more results on Choquet integrals with minimal assumptions on the monotone set function through which they are defined we refer to 
\cite{PonceSpector23}.

Let $\kappa \in [0, n)$.
If  $f \in L^1_{\loc}(\Rn)$, then the fractional centred Hardy-Littlewood maximal function of $f$ is  written as 
\[
M_\kappa f(x) := \sup_{r>0} r^{\kappa-n} \int_{B(x, r)} |f(y)| \, dy.
\]
 The non-fractional  centred maximal function $M_0 f$ is written as $Mf$.
If $f$ is defined only on $\Omega$, then 
we agree that $f(x)=0$ for all $x\in\Rn\backslash\Omega$
 in  the definition of $M_\kappa $.

Adams showed that the fractional maximal function is bounded, \cite[Theorem 7(a)]{Adams1998}.
 %We take an agreement that  $\delta /0:=\infty$. %Tämä on jo johdannossa.

\begin{theorem}[Theorem 7(a) of \cite{Adams1998}]\label{thm:fractional-maximal-function}
Suppose that $\delta \in (0, n]$ and $\kappa \in [0, n)$.  If $p \in (\delta/n, \delta /\kappa)$,
then there exists a constant $c$ depending only on $n$, $\delta$, $\kappa$, and $p$ such that
 for every $f \in L^1_{\loc}(\Rn)$ we have
\[  
\int_\Rn (M_\kappa f(x))^p \, d \Ha^{\delta-\kappa p}_\infty \le c \int_\Rn |f(x)|^p \, d \Ha^{\delta}_\infty.
\]
\end{theorem}

%%%%%%%%%%%%%%%%%%%%%%%%%%%%%%%%%%%%%%%%%%%%%%%%%%%%%%%%%%%%%%%%%%%%%%%%%%%%%%%%%%%
%%%%%%%%%%%%%%%%%%%%%%%%%%%%%%%%%%%%%%%%%%%%%%%%%%%%%%%%%%%%%%%%%%%%%%%%%%%%%%%%%%%
%%%%%%%%%%%%%%%%%%%%%%%%%%%%%%%%%%%%%%%%%%%%%%%%%%%%%%
%%%%%%%%%%%%%%%%%%%%%%%%%%%%%%%%%%%%%%%%%%%%%%%%%%%%%%%%%%%%%%%%%%%%

\section{On the  Hardy inequality}\label{Section-Hardy}

When $\Omega$ is a bounded, open, connected 
set in the Euclidean $n$-space with some regularity conditions,
we prove the following Hardy-type inequality for compactly supported, continuously differentiable  functions $u$ defined on  $\Omega$. 
Whenever $\delta \in (0, n]$, $\kappa \in [0, 1)$, and $p\in(\delta /n, \delta /\kappa )$ there exists a constant depending only on 
$n$, $\delta$, $\kappa$, $p$, and
the domain $\Omega$ such that
the inequality
\begin{equation*}
\int_{\Omega}\frac{\vert u(x)\vert^p}{\dist (x,\partial \Omega)^{p(1-\kappa )}}\,d\Ha_{\infty}^{\delta- \kappa p}\le
c\int_{\Omega}\vert \nabla u(x)\vert^p\,d\Ha_{\infty}^{\delta}
\end{equation*}
holds 
for all $u\in C^{1}_0(\Omega )$, that is continuously differentiable functions on $\Omega$ with compact support.
Whenever $\delta =n$ and $\kappa =0$, this inequality recovers the classical Hardy inequality for domains with certain regularity properties.

We recall the notion of outer regularity of a set.

\begin{definition}
Let $\Omega$ be an open, proper subset of  $\Rn$ and $\Rn\backslash\Omega=:\Omega^c$ its complement with respect to $\Rn$.
If there exists  a constant $b>0$ 
such that for every $y\in\partial \Omega$ and all $r>0$,
\begin{equation}\label{plump}
\vert B(y,r)\cap \Omega^c\vert\geq b\vert B(y,r)\vert\,,
\end{equation}
then  $\Omega$   has an outer regularity property.
\end{definition}

Examples of domains which have an outer regularity property are  domains with  a plump complement i.e. if there exists $\sigma>0$ such that for all $y\in\partial \Omega$ and all $t\in (0,\sigma ]$, there is an 
$x\in (\Rn\backslash\Omega )\cap \overline{B(y,t)}$ with $\dist (x, \partial \Omega )\geq bt$. See also Figure~\ref{fig:outer-regularity}.

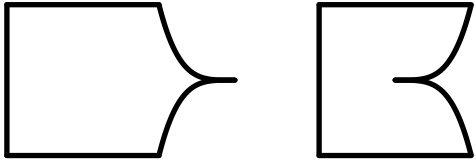
\begin{figure}
\begin{tikzpicture}[line cap=round,line join=round,>=triangle 45,x=1.0cm,y=1.0cm]
\clip(-2.7,-0.6) rectangle (3.8,2.2);
\draw [line width=2.pt] (0.,2.)-- (2.,2.);
\draw [line width=2.pt] (2.,0.)-- (0.,0.);
\draw [line width=2.pt] (0.,0.)-- (0.,2.);
\draw[line width=2.pt,smooth,samples=100,domain=2:3] plot(\x,{((\x)-3)^(4)+1});
\draw[line width=2.pt,smooth,samples=100,domain=2:3] plot(\x,{0-(((\x)-3)^(4)+1)+2});
\end{tikzpicture}
\begin{tikzpicture}[line cap=round,line join=round,>=triangle 45,x=1.0cm,y=1.0cm]
\clip(-0.2,-0.6) rectangle (3.0,2.2);
\draw [line width=2.pt] (0.,2.)-- (2.,2.);
\draw [line width=2.pt] (2.,0.)-- (0.,0.);
\draw [line width=2.pt] (0.,0.)-- (0.,2.);
\draw[line width=2.pt,smooth,samples=100,domain=1:2] plot(\x,{((\x)-1)^(4.0)+1});
\draw[line width=2.pt,smooth,samples=100,domain=1:2] plot(\x,{0-(((\x)-1)^(4.0)+1)+2});
\end{tikzpicture}
\caption{The left-hand side domain satisfies the outer regularity condition, but the right-hand  side one does not.}\label{fig:outer-regularity}
\end{figure}

The proof for the  Hardy inequality  relies on 
a pointwise inequality by P.\ Haj{\l}asz \cite[Proposition 1]{Hajlasz1999}.
We modify his proof by using the fractional maximal operator. This idea
 of  using the fractional maximal operator  goes back to  Adams \cite{Adams1975}.

\begin{proposition}
Let $\Omega$ be an open and proper subset of $\Rn$, that has the outer regularity property with a constant $b$.
Then for every $u \in C^1_0(\Omega)$, we have
\begin{equation}\label{Hajlasz_pointwise}
\vert u(x)\vert \le c(n, b, \kappa)\dist(x, \partial \Omega)^{1-\kappa} M_\kappa |\nabla u|(x)
\end{equation}
for all $x\in\Omega$.
\end{proposition}

\begin{proof}
Let $u \in C^1_0(\Omega)$. Fix $x \in \Omega$, and let $x^{*}\in\partial\Omega$ such that $\vert x-x^{*}\vert=\dist (x,\partial \Omega )$.
We write $B_x:= B(x, 2|x-x^*|)$. Then  \cite[Lemma 7.16]{GT}  yields     for every $y \in B_x \cap \Omega^c$ 
\[
\begin{split}
|u(x)| &= |u(x) - u(y)| \le |u(x) - u_{B_x}| + |u(y)-u_{B_x}|\\
&\le \int_{B_x} \frac{|\nabla u(z)|}{|x-z|^{n-1}} \, dz + \int_{B_x} \frac{|\nabla u(z)|}{|y-z|^{n-1}} \, dz\,.
\end{split}
\]
Next we show that
\begin{equation}\label{ovela}
\inf_{y \in {B_x} \cap \Omega^c} \int_{B_x} \frac{|\nabla u(z)|}{|y-z|^{n-1}} \, dz
\le c(n, b) \int_{B_x} \frac{|\nabla u(z)|}{|x-z|^{n-1}} \, dz.
\end{equation}
By Fubini's theorem 
\begin{equation*}
\fint_{B_x \cap \Omega^c}\int_{B_x} \frac{|\nabla u(z)|}{|y-z|^{n-1}} \, dz \, dy
=\int_{B_x}  \fint_{B_x \cap \Omega^c}\frac{|\nabla u(z)|}{|y-z|^{n-1}}  \, dy \, dz\,.
\end{equation*}
Since $B(x^*, \vert x-x^* \vert )\cap\Omega^c) \subset B_x\cap \Omega^c$,
the  outer regularity assumption can be used for $\vert B(x^*, \vert x-x^*)\cap\Omega^c\vert$, and we obtain
\[
\int_{B_x}  \fint_{B_x \cap \Omega^c}\frac{|\nabla u(z)|}{|y-z|^{n-1}}  \, dy \, dz
 \le c(n, b) \frac{1}{\vert B(x^*,\vert x-x^*\vert )\vert} \int_{B_x}  |\nabla u(z)| \int_{B_x} |y-z|^{1-n}  \, dy \, dz\,.
\]
Now  \cite[Lemma 3.11.3]{Ziemer89} yields that
\begin{equation*}
 \frac{1}{\vert B(x^*,\vert x-x^*\vert )\vert} \int_{B_x}  |\nabla u(z)| \int_{B_x} |y-z|^{1-n}  \, dy \, dz
 \le c(n, b) \int_{B_x}  \frac{|\nabla u(z)|}{|x-z|^{n-1}} \, dz\,.
\end{equation*}
Hence, equation \eqref{ovela} follows.
Thus we have  for every $x\in\Omega$ 
\[
|u(x)| 
\le c(n, b)  \int_{B(x, 2|x-x^*|)}  \frac{|\nabla u(z)|}{|x-z|^{n-1}} \, dz.
\]

Then we estimate the Riesz potential by the fractional maximal function.
Writing  $A_j =\{y \in \Rn : 2^{-j} r \le |x-y|< 2^{-j+1}r\}$ with $r=2|x-x^*|$ yields that
\[
\begin{split}
\int_{B_x} \frac{|\nabla u(y)|}{|x-z|^{n-1}} \, dy 
&= \sum_{j=1}^\infty \int_{A_j} \frac{|\nabla u(y)|}{|x-z|^{n-1}} \, dy\\
&\le  \sum_{j=1}^\infty (2^{-j}r)^{1-n} \int_{B(x, 2^{-j+1}r)} |\nabla u(y)| \, dy\\
&\le   c(n, \kappa) r^{1-\kappa} M_\kappa |\nabla u|(x)\,.
\end{split}
\]
In the last step  the sum of a geometric series is used.
\end{proof}

\begin{theorem}\label{Hardy-theorem}
Let $\Omega$ be an open and proper subset of $\Rn$, that has the outer regularity property with a constant $b$.
Let $\delta \in (0, n]$, $\kappa \in [0, 1)$,  $p\in (\delta /n, \delta/\kappa )$.
Then there exists a positive constant $c <\infty$  depending only on  $n$,  $\delta$, $\kappa$, 
 $p$, and the outer regularity constant $b$ such that
\begin{equation*}
\int_{\Omega}\frac{\vert u(x)\vert^p}{\dist (x,\partial \Omega)^{p(1-\kappa)}}\,d\Ha_{\infty}^{\delta-\kappa p}\le
c\int_{\Omega}\vert \nabla u(x)\vert^p\,d\Ha_{\infty}^{\delta}
\end{equation*}
for all $u\in C^{1}_0(\Omega )$.
\end{theorem}

\begin{proof}
Taking both sides of inequality
(\ref{Hajlasz_pointwise}) to the power $p$ and integrating in the Choquet sense with respect to the Hausdorff content 
and using the boundedness result of the fractional maximal function with respect to the Hausdorff content, 
Theorem~\ref{thm:fractional-maximal-function}, imply

\begin{equation*}
\begin{split}
\int_{\Omega}\frac{\vert u(x)\vert^{p}}{\dist (x,\partial \Omega)^{p(1-\kappa)}}\,d\Ha_{\infty}^{\delta-\kappa p}&\le
c \int_{\Omega}  (M_\kappa |\nabla u|(x))^p \,d\Ha_{\infty}^{\delta-\kappa p}\\
&\le
c \int_{\Omega}\vert \nabla u(x)\vert^p\,d\Ha_{\infty}^{\delta}\,. \qedhere
\end{split}
\end{equation*}
\end{proof}

Using Lemma \ref{GeneralizationOV} together Theorem \ref{Hardy-theorem}  yields the following result.

\begin{corollary}
Let $\Omega$ be an open and proper subset of $\Rn$, that has the outer regularity property. 
Let $\delta \in (0, n]$, $\kappa \in [0, 1)$,  $p\in (\delta /n, \delta/\kappa )$.
If $\epsilon\in (0, \delta )$
then there exists a positive constant $c <\infty$  depending only on  $n$,  $\delta$, $\kappa$,  $p$, 
$\epsilon$,  and the outer regularity constant $b$ such that
\begin{equation*}
\int_{\Omega}\frac{\vert u(x)\vert^p}{\dist (x,\partial \Omega)^{p(1-\kappa)}}\,d\Ha_{\infty}^{\delta-\kappa p}\le
c
\Biggl(
\int_{\Omega}\vert \nabla u(x)\vert^{\frac{p(\delta-\epsilon)}{\delta}}d\Ha_{\infty}^{\delta-\epsilon}
\Biggr)^{\frac{\delta}{p(\delta -\epsilon )}}
\end{equation*}
for all $u\in C^{1}_0(\Omega )$.
\end{corollary}
This means that the exponent of the function on the left-hand side could be  strictly larger than the exponent of the absolute value of its gradient on the right-hand side.
But this affects to the dimension of the Hausdorff content though.

In particular, Theorem \ref{Hardy-theorem} gives Corollary \ref{HardyCorollary}.

%%%%%%%%%%%%%%%%%%%%%%%%%%%%%%%%%%%%%%%%%%%%%%%%%%%%%%%%%%%%%%%%%%%%%%%%%%%%%%%%%
%%%%%%%%%%%%%%%%%%%%%%%%%%%%%%%%%%%%%%%%%%%%%%%%%%%%%%%%%%%%%%%%%%%%%%%%%%%%%%%%%
%%%%%%%%%%%%%%%%%%%%%%%%%%%%%%%%%%%%%%%%%%%%%%%%%%%%%%%%%%%%%%%%%%%%%%%%%%%%%%%%%
\section{ On inequalities for functions defined on $s$-John domains}\label{sJohn}

John domains
do not allow outward  $n$-dimensional $s$-cusps, $s>1$, that is $s$-spires.
An $s$-spire, $s>1$  is defined to be the set
\begin{equation}\label{spire}
\Omega =\{(x_1, x_2, \dots , x_n)\in (0,1)\times \R^{n-1}\, : \vert\vert (x_2, x_3, \dots ,x_n)\vert\vert < x^s\}\,.
\end{equation}
This domain is called a  power cusp in \cite[Section 5.6, p. 308]{MazyaPoborchi}.
It  is also  called  an  $n$-dimensional cusp, and it  belongs to a class of cuspidal domains
\cite{ADLG2013}. 
The  power cusp defined in \eqref{spire}
 is not  a John domain, but it is an example of $s$-John domains.
In order to consider cuspidal domains  we 
recall a definition of {$s$-John domains}.
\begin{definition}\label{s-john}
Suppose that  $\Omega$ is a bounded domain in $ \R^n$, $n\geq 2$. 
Let $s\ge1$.
The domain $\Omega$ is an $s$-John domain if there exist  finite, positive constants $c, C$ and a point $x_0 \in \Omega$ such that each point $x\in \Omega$ can be joined to $x_0$ by a 
rectifiable curve $\gamma_x:[0,\ell(\gamma_x)] \to \Omega$, parametrised by its arc length, such that $\gamma_x(0) = x$, $\gamma_x(\ell(\gamma_x)) = x_0$, $\ell(\gamma_x)\leq C\,,$ and
\[
\dist\big(\gamma_x(t), \partial \Omega \big)
\geq c t^s
 \quad \text{for all} \quad t\in[0, \ell(\gamma_x)].
\]
\end{definition}

The point $x_0$ is called a John centre and constants $c$ and $C$ are John constants.
If $s=1$, this definition reduces to the classical John domain definition.
John domains
are  called $1$-John domains in Definition \ref{s-john}.
One can construct Nikod\'ym-type domains where $s$-John property fails
\cite[1.1.4]{Mazya}.
Closely related concepts are defined in \cite[4.11.2]{Mazya}, that is the classes
${\mathcal{I}}^{n-1}_{p,\alpha}$ and ${\mathcal{J}}^{n-1}_{p,\alpha}$.

\begin{figure}
\begin{tikzpicture}[line cap=round,line join=round,>=triangle 45,x=1.0cm,y=1.0cm]
\clip(-1.12,-0.8) rectangle (2.1,2.1);
\draw [line width=2.pt] (0.,2.)-- (2.,2.);
\draw [line width=2.pt] (2.,0.)-- (0.,0.);
\draw [line width=2.pt] (0.,0.)-- (0.,0.5);
\draw [line width=2.pt] (0.,1.5)-- (0.,2);
\draw [line width=2.pt] (0.,0.5)-- (-1.,1);
\draw [line width=2.pt] (0.,1.5)-- (-1.,1);
\draw [line width=2.pt] (0.5,1.5)-- (0.5,0.5);
\draw[line width=2.pt,smooth,samples=100,domain=1:2] plot(\x,{((\x)-1)^(4.0)+1});
\draw[line width=2.pt,smooth,samples=100,domain=1:2] plot(\x,{0-(((\x)-1)^(4.0)+1)+2});
\end{tikzpicture}
\begin{tikzpicture}[line cap=round,line join=round,>=triangle 45,x=0.7cm,y=0.7cm]
\clip(-2.1,-2.5) rectangle (2.1,2.1);
\fill[line width=0.8pt] (0.9105222500000001,1.) -- (1.,1.) -- (1.,0.91052225) -- (0.9105222500000006,0.9105222500000009) -- cycle;
\fill[line width=0.8pt] (0.9105222500000005,0.8374277499999995) -- (1.,0.8374277499999999) -- (1.,0.74795) -- (0.9105222500000003,0.74795) -- cycle;
\fill[line width=0.8pt] (0.9105222500000003,0.54205) -- (1.,0.54205) -- (1.,0.45257225000000034) -- (0.9105222500000006,0.4525722500000008) -- cycle;
\fill[line width=0.8pt] (0.9105222500000008,0.3794777499999994) -- (1.,0.37947775000000034) -- (1.,0.29) -- (0.9105222500000003,0.29) -- cycle;
\fill[line width=0.8pt] (0.8374277499999997,0.37947774999999967) -- (0.74795,0.37947775) -- (0.74795,0.29) -- (0.8374277500000002,0.29) -- cycle;
\fill[line width=0.8pt] (0.8374277499999995,0.4525722500000007) -- (0.74795,0.45257225) -- (0.74795,0.54205) -- (0.8374277500000002,0.54205) -- cycle;
\fill[line width=0.8pt] (0.8374277500000002,0.74795) -- (0.74795,0.74795) -- (0.74795,0.8374277500000001) -- (0.8374277499999994,0.8374277499999996) -- cycle;
\fill[line width=0.8pt] (0.74795,0.9105222500000002) -- (0.74795,1.) -- (0.83742775,1.) -- (0.8374277499999996,0.9105222500000005) -- cycle;
\fill[line width=0.8pt] (0.29,1.) -- (0.37947774999999995,1.) -- (0.3794777499999991,0.9105222500000005) -- (0.29,0.9105222500000001) -- cycle;
\fill[line width=0.8pt] (0.45257225,1.) -- (0.54205,1.) -- (0.54205,0.9105222500000003) -- (0.4525722500000005,0.9105222500000005) -- cycle;
\fill[line width=0.8pt] (0.29,0.8374277499999999) -- (0.37947774999999945,0.8374277499999996) -- (0.37947774999999984,0.74795) -- (0.29,0.74795) -- cycle;
\fill[line width=0.8pt] (0.4525722500000004,0.8374277499999994) -- (0.54205,0.8374277500000002) -- (0.54205,0.74795) -- (0.4525722499999999,0.74795) -- cycle;
\fill[line width=0.8pt] (0.29,0.54205) -- (0.3794777499999999,0.54205) -- (0.37947774999999956,0.45257225000000045) -- (0.29,0.45257225) -- cycle;
\fill[line width=0.8pt] (0.45257224999999995,0.54205) -- (0.54205,0.54205) -- (0.54205,0.4525722499999998) -- (0.4525722500000005,0.4525722500000007) -- cycle;
\fill[line width=0.8pt] (0.4525722500000006,0.3794777499999995) -- (0.54205,0.3794777499999997) -- (0.54205,0.29) -- (0.4525722500000002,0.29) -- cycle;
\fill[line width=0.8pt] (0.29,0.37947774999999995) -- (0.3794777499999992,0.3794777499999994) -- (0.3794777500000001,0.29) -- (0.29,0.29) -- cycle;
\fill[line width=0.8pt] (0.9105222500000001,-1.) -- (1.,-1.) -- (1.,-0.91052225) -- (0.9105222500000006,-0.9105222500000009) -- cycle;
\fill[line width=0.8pt] (0.9105222500000005,-0.8374277499999995) -- (1.,-0.8374277499999999) -- (1.,-0.74795) -- (0.9105222500000003,-0.74795) -- cycle;
\fill[line width=0.8pt] (0.9105222500000003,-0.54205) -- (1.,-0.54205) -- (1.,-0.45257225000000034) -- (0.9105222500000006,-0.4525722500000008) -- cycle;
\fill[line width=0.8pt] (0.9105222500000008,-0.3794777499999994) -- (1.,-0.37947775000000034) -- (1.,-0.29) -- (0.9105222500000003,-0.29) -- cycle;
\fill[line width=0.8pt] (0.8374277499999997,-0.37947774999999967) -- (0.74795,-0.37947775) -- (0.74795,-0.29) -- (0.8374277500000002,-0.29) -- cycle;
\fill[line width=0.8pt] (0.8374277499999995,-0.4525722500000007) -- (0.74795,-0.45257225) -- (0.74795,-0.54205) -- (0.8374277500000002,-0.54205) -- cycle;
\fill[line width=0.8pt] (0.8374277500000002,-0.74795) -- (0.74795,-0.74795) -- (0.74795,-0.8374277500000001) -- (0.8374277499999994,-0.8374277499999996) -- cycle;
\fill[line width=0.8pt] (0.74795,-0.9105222500000002) -- (0.74795,-1.) -- (0.83742775,-1.) -- (0.8374277499999996,-0.9105222500000005) -- cycle;
\fill[line width=0.8pt] (0.29,-1.) -- (0.37947774999999995,-1.) -- (0.3794777499999991,-0.9105222500000005) -- (0.29,-0.9105222500000001) -- cycle;
\fill[line width=0.8pt] (0.45257225,-1.) -- (0.54205,-1.) -- (0.54205,-0.9105222500000003) -- (0.4525722500000005,-0.9105222500000005) -- cycle;
\fill[line width=0.8pt] (0.29,-0.8374277499999999) -- (0.37947774999999945,-0.8374277499999996) -- (0.37947774999999984,-0.74795) -- (0.29,-0.74795) -- cycle;
\fill[line width=0.8pt] (0.4525722500000004,-0.8374277499999994) -- (0.54205,-0.8374277500000002) -- (0.54205,-0.74795) -- (0.4525722499999999,-0.74795) -- cycle;
\fill[line width=0.8pt] (0.29,-0.54205) -- (0.3794777499999999,-0.54205) -- (0.37947774999999956,-0.45257225000000045) -- (0.29,-0.45257225) -- cycle;
\fill[line width=0.8pt] (0.45257224999999995,-0.54205) -- (0.54205,-0.54205) -- (0.54205,-0.4525722499999998) -- (0.4525722500000005,-0.4525722500000007) -- cycle;
\fill[line width=0.8pt] (0.4525722500000006,-0.3794777499999995) -- (0.54205,-0.3794777499999997) -- (0.54205,-0.29) -- (0.4525722500000002,-0.29) -- cycle;
\fill[line width=0.8pt] (0.29,-0.37947774999999995) -- (0.3794777499999992,-0.3794777499999994) -- (0.3794777500000001,-0.29) -- (0.29,-0.29) -- cycle;
\fill[line width=0.8pt] (-0.91052225,1.) -- (-1.,1.) -- (-1.,0.9105222500000001) -- (-0.9105222500000005,0.910522250000001) -- cycle;
\fill[line width=0.8pt] (-0.9105222500000004,0.8374277499999996) -- (-1.,0.83742775) -- (-1.,0.74795) -- (-0.9105222500000002,0.74795) -- cycle;
\fill[line width=0.8pt] (-0.9105222500000002,0.54205) -- (-1.,0.54205) -- (-1.,0.45257225000000045) -- (-0.9105222500000006,0.4525722500000009) -- cycle;
\fill[line width=0.8pt] (-0.9105222500000008,0.3794777499999995) -- (-1.,0.37947775000000045) -- (-1.,0.29) -- (-0.9105222500000003,0.29) -- cycle;
\fill[line width=0.8pt] (-0.8374277499999997,0.3794777499999998) -- (-0.74795,0.3794777500000001) -- (-0.74795,0.29) -- (-0.8374277500000002,0.29) -- cycle;
\fill[line width=0.8pt] (-0.8374277499999995,0.4525722500000008) -- (-0.74795,0.4525722500000001) -- (-0.74795,0.54205) -- (-0.8374277500000001,0.54205) -- cycle;
\fill[line width=0.8pt] (-0.8374277500000001,0.74795) -- (-0.74795,0.74795) -- (-0.74795,0.8374277500000002) -- (-0.8374277499999992,0.8374277499999997) -- cycle;
\fill[line width=0.8pt] (-0.74795,0.9105222500000003) -- (-0.74795,1.) -- (-0.8374277499999999,1.) -- (-0.8374277499999995,0.9105222500000006) -- cycle;
\fill[line width=0.8pt] (-0.29,1.) -- (-0.37947774999999984,1.) -- (-0.379477749999999,0.9105222500000005) -- (-0.29,0.9105222500000001) -- cycle;
\fill[line width=0.8pt] (-0.4525722499999999,1.) -- (-0.54205,1.) -- (-0.54205,0.9105222500000004) -- (-0.4525722500000004,0.9105222500000005) -- cycle;
\fill[line width=0.8pt] (-0.29,0.8374277499999999) -- (-0.37947774999999934,0.8374277499999996) -- (-0.3794777499999997,0.74795) -- (-0.29,0.74795) -- cycle;
\fill[line width=0.8pt] (-0.4525722500000003,0.8374277499999994) -- (-0.54205,0.8374277500000004) -- (-0.54205,0.74795) -- (-0.4525722499999998,0.74795) -- cycle;
\fill[line width=0.8pt] (-0.29,0.54205) -- (-0.37947774999999984,0.54205) -- (-0.3794777499999995,0.4525722500000005) -- (-0.29,0.45257225000000006) -- cycle;
\fill[line width=0.8pt] (-0.4525722499999999,0.54205) -- (-0.54205,0.54205) -- (-0.54205,0.45257224999999984) -- (-0.45257225000000045,0.45257225000000073) -- cycle;
\fill[line width=0.8pt] (-0.45257225000000056,0.37947774999999956) -- (-0.54205,0.3794777499999998) -- (-0.54205,0.29) -- (-0.4525722500000001,0.29) -- cycle;
\fill[line width=0.8pt] (-0.29,0.37947775) -- (-0.37947774999999917,0.37947774999999945) -- (-0.37947775000000006,0.29) -- (-0.29,0.29) -- cycle;
\fill[line width=0.8pt] (-0.91052225,-1.) -- (-1.,-1.) -- (-1.,-0.9105222500000001) -- (-0.9105222500000005,-0.910522250000001) -- cycle;
\fill[line width=0.8pt] (-0.9105222500000004,-0.8374277499999996) -- (-1.,-0.83742775) -- (-1.,-0.74795) -- (-0.9105222500000002,-0.74795) -- cycle;
\fill[line width=0.8pt] (-0.9105222500000002,-0.54205) -- (-1.,-0.54205) -- (-1.,-0.45257225000000045) -- (-0.9105222500000006,-0.4525722500000009) -- cycle;
\fill[line width=0.8pt] (-0.9105222500000008,-0.3794777499999995) -- (-1.,-0.37947775000000045) -- (-1.,-0.29) -- (-0.9105222500000003,-0.29) -- cycle;
\fill[line width=0.8pt] (-0.8374277499999997,-0.3794777499999998) -- (-0.74795,-0.3794777500000001) -- (-0.74795,-0.29) -- (-0.8374277500000002,-0.29) -- cycle;
\fill[line width=0.8pt] (-0.8374277499999995,-0.4525722500000008) -- (-0.74795,-0.4525722500000001) -- (-0.74795,-0.54205) -- (-0.8374277500000001,-0.54205) -- cycle;
\fill[line width=0.8pt] (-0.8374277500000001,-0.74795) -- (-0.74795,-0.74795) -- (-0.74795,-0.8374277500000002) -- (-0.8374277499999992,-0.8374277499999997) -- cycle;
\fill[line width=0.8pt] (-0.74795,-0.9105222500000003) -- (-0.74795,-1.) -- (-0.8374277499999999,-1.) -- (-0.8374277499999995,-0.9105222500000006) -- cycle;
\fill[line width=0.8pt] (-0.29,-1.) -- (-0.37947774999999984,-1.) -- (-0.379477749999999,-0.9105222500000005) -- (-0.29,-0.9105222500000001) -- cycle;
\fill[line width=0.8pt] (-0.4525722499999999,-1.) -- (-0.54205,-1.) -- (-0.54205,-0.9105222500000004) -- (-0.4525722500000004,-0.9105222500000005) -- cycle;
\fill[line width=0.8pt] (-0.29,-0.8374277499999999) -- (-0.37947774999999934,-0.8374277499999996) -- (-0.3794777499999997,-0.74795) -- (-0.29,-0.74795) -- cycle;
\fill[line width=0.8pt] (-0.4525722500000003,-0.8374277499999994) -- (-0.54205,-0.8374277500000004) -- (-0.54205,-0.74795) -- (-0.4525722499999998,-0.74795) -- cycle;
\fill[line width=0.8pt] (-0.29,-0.54205) -- (-0.37947774999999984,-0.54205) -- (-0.3794777499999995,-0.4525722500000005) -- (-0.29,-0.45257225000000006) -- cycle;
\fill[line width=0.8pt] (-0.4525722499999999,-0.54205) -- (-0.54205,-0.54205) -- (-0.54205,-0.45257224999999984) -- (-0.45257225000000045,-0.45257225000000073) -- cycle;
\fill[line width=0.8pt] (-0.45257225000000056,-0.37947774999999956) -- (-0.54205,-0.3794777499999998) -- (-0.54205,-0.29) -- (-0.4525722500000001,-0.29) -- cycle;
\fill[line width=0.8pt] (-0.29,-0.37947775) -- (-0.37947774999999917,-0.37947774999999945) -- (-0.37947775000000006,-0.29) -- (-0.29,-0.29) -- cycle;
\draw [line width=0.8pt] (0.9105222500000001,1.)-- (1.,1.);
\draw [line width=0.8pt] (1.,1.)-- (1.,0.91052225);
\draw [line width=0.8pt] (1.,0.91052225)-- (0.9105222500000006,0.9105222500000009);
\draw [line width=0.8pt] (0.9105222500000006,0.9105222500000009)-- (0.9105222500000001,1.);
\draw [line width=0.8pt] (0.9105222500000005,0.8374277499999995)-- (1.,0.8374277499999999);
\draw [line width=0.8pt] (1.,0.8374277499999999)-- (1.,0.74795);
\draw [line width=0.8pt] (1.,0.74795)-- (0.9105222500000003,0.74795);
\draw [line width=0.8pt] (0.9105222500000003,0.74795)-- (0.9105222500000005,0.8374277499999995);
\draw [line width=0.8pt] (0.9105222500000003,0.54205)-- (1.,0.54205);
\draw [line width=0.8pt] (1.,0.54205)-- (1.,0.45257225000000034);
\draw [line width=0.8pt] (1.,0.45257225000000034)-- (0.9105222500000006,0.4525722500000008);
\draw [line width=0.8pt] (0.9105222500000006,0.4525722500000008)-- (0.9105222500000003,0.54205);
\draw [line width=0.8pt] (0.9105222500000008,0.3794777499999994)-- (1.,0.37947775000000034);
\draw [line width=0.8pt] (1.,0.37947775000000034)-- (1.,0.29);
\draw [line width=0.8pt] (1.,0.29)-- (0.9105222500000003,0.29);
\draw [line width=0.8pt] (0.9105222500000003,0.29)-- (0.9105222500000008,0.3794777499999994);
\draw [line width=0.8pt] (0.8374277499999997,0.37947774999999967)-- (0.74795,0.37947775);
\draw [line width=0.8pt] (0.74795,0.37947775)-- (0.74795,0.29);
\draw [line width=0.8pt] (0.74795,0.29)-- (0.8374277500000002,0.29);
\draw [line width=0.8pt] (0.8374277500000002,0.29)-- (0.8374277499999997,0.37947774999999967);
\draw [line width=0.8pt] (0.8374277499999995,0.4525722500000007)-- (0.74795,0.45257225);
\draw [line width=0.8pt] (0.74795,0.45257225)-- (0.74795,0.54205);
\draw [line width=0.8pt] (0.74795,0.54205)-- (0.8374277500000002,0.54205);
\draw [line width=0.8pt] (0.8374277500000002,0.54205)-- (0.8374277499999995,0.4525722500000007);
\draw [line width=0.8pt] (0.8374277500000002,0.74795)-- (0.74795,0.74795);
\draw [line width=0.8pt] (0.74795,0.74795)-- (0.74795,0.8374277500000001);
\draw [line width=0.8pt] (0.74795,0.8374277500000001)-- (0.8374277499999994,0.8374277499999996);
\draw [line width=0.8pt] (0.8374277499999994,0.8374277499999996)-- (0.8374277500000002,0.74795);
\draw [line width=0.8pt] (0.74795,0.9105222500000002)-- (0.74795,1.);
\draw [line width=0.8pt] (0.74795,1.)-- (0.83742775,1.);
\draw [line width=0.8pt] (0.83742775,1.)-- (0.8374277499999996,0.9105222500000005);
\draw [line width=0.8pt] (0.8374277499999996,0.9105222500000005)-- (0.74795,0.9105222500000002);
\draw [line width=0.8pt] (0.29,1.)-- (0.37947774999999995,1.);
\draw [line width=0.8pt] (0.37947774999999995,1.)-- (0.3794777499999991,0.9105222500000005);
\draw [line width=0.8pt] (0.3794777499999991,0.9105222500000005)-- (0.29,0.9105222500000001);
\draw [line width=0.8pt] (0.29,0.9105222500000001)-- (0.29,1.);
\draw [line width=0.8pt] (0.45257225,1.)-- (0.54205,1.);
\draw [line width=0.8pt] (0.54205,1.)-- (0.54205,0.9105222500000003);
\draw [line width=0.8pt] (0.54205,0.9105222500000003)-- (0.4525722500000005,0.9105222500000005);
\draw [line width=0.8pt] (0.4525722500000005,0.9105222500000005)-- (0.45257225,1.);
\draw [line width=0.8pt] (0.29,0.8374277499999999)-- (0.37947774999999945,0.8374277499999996);
\draw [line width=0.8pt] (0.37947774999999945,0.8374277499999996)-- (0.37947774999999984,0.74795);
\draw [line width=0.8pt] (0.37947774999999984,0.74795)-- (0.29,0.74795);
\draw [line width=0.8pt] (0.29,0.74795)-- (0.29,0.8374277499999999);
\draw [line width=0.8pt] (0.4525722500000004,0.8374277499999994)-- (0.54205,0.8374277500000002);
\draw [line width=0.8pt] (0.54205,0.8374277500000002)-- (0.54205,0.74795);
\draw [line width=0.8pt] (0.54205,0.74795)-- (0.4525722499999999,0.74795);
\draw [line width=0.8pt] (0.4525722499999999,0.74795)-- (0.4525722500000004,0.8374277499999994);
\draw [line width=0.8pt] (0.29,0.54205)-- (0.3794777499999999,0.54205);
\draw [line width=0.8pt] (0.3794777499999999,0.54205)-- (0.37947774999999956,0.45257225000000045);
\draw [line width=0.8pt] (0.37947774999999956,0.45257225000000045)-- (0.29,0.45257225);
\draw [line width=0.8pt] (0.29,0.45257225)-- (0.29,0.54205);
\draw [line width=0.8pt] (0.45257224999999995,0.54205)-- (0.54205,0.54205);
\draw [line width=0.8pt] (0.54205,0.54205)-- (0.54205,0.4525722499999998);
\draw [line width=0.8pt] (0.54205,0.4525722499999998)-- (0.4525722500000005,0.4525722500000007);
\draw [line width=0.8pt] (0.4525722500000005,0.4525722500000007)-- (0.45257224999999995,0.54205);
\draw [line width=0.8pt] (0.4525722500000006,0.3794777499999995)-- (0.54205,0.3794777499999997);
\draw [line width=0.8pt] (0.54205,0.3794777499999997)-- (0.54205,0.29);
\draw [line width=0.8pt] (0.54205,0.29)-- (0.4525722500000002,0.29);
\draw [line width=0.8pt] (0.4525722500000002,0.29)-- (0.4525722500000006,0.3794777499999995);
\draw [line width=0.8pt] (0.29,0.37947774999999995)-- (0.3794777499999992,0.3794777499999994);
\draw [line width=0.8pt] (0.3794777499999992,0.3794777499999994)-- (0.3794777500000001,0.29);
\draw [line width=0.8pt] (0.3794777500000001,0.29)-- (0.29,0.29);
\draw [line width=0.8pt] (0.29,0.29)-- (0.29,0.37947774999999995);
\draw [line width=0.8pt] (0.9105222500000001,-1.)-- (1.,-1.);
\draw [line width=0.8pt] (1.,-1.)-- (1.,-0.91052225);
\draw [line width=0.8pt] (1.,-0.91052225)-- (0.9105222500000006,-0.9105222500000009);
\draw [line width=0.8pt] (0.9105222500000006,-0.9105222500000009)-- (0.9105222500000001,-1.);
\draw [line width=0.8pt] (0.9105222500000005,-0.8374277499999995)-- (1.,-0.8374277499999999);
\draw [line width=0.8pt] (1.,-0.8374277499999999)-- (1.,-0.74795);
\draw [line width=0.8pt] (1.,-0.74795)-- (0.9105222500000003,-0.74795);
\draw [line width=0.8pt] (0.9105222500000003,-0.74795)-- (0.9105222500000005,-0.8374277499999995);
\draw [line width=0.8pt] (0.9105222500000003,-0.54205)-- (1.,-0.54205);
\draw [line width=0.8pt] (1.,-0.54205)-- (1.,-0.45257225000000034);
\draw [line width=0.8pt] (1.,-0.45257225000000034)-- (0.9105222500000006,-0.4525722500000008);
\draw [line width=0.8pt] (0.9105222500000006,-0.4525722500000008)-- (0.9105222500000003,-0.54205);
\draw [line width=0.8pt] (0.9105222500000008,-0.3794777499999994)-- (1.,-0.37947775000000034);
\draw [line width=0.8pt] (1.,-0.37947775000000034)-- (1.,-0.29);
\draw [line width=0.8pt] (1.,-0.29)-- (0.9105222500000003,-0.29);
\draw [line width=0.8pt] (0.9105222500000003,-0.29)-- (0.9105222500000008,-0.3794777499999994);
\draw [line width=0.8pt] (0.8374277499999997,-0.37947774999999967)-- (0.74795,-0.37947775);
\draw [line width=0.8pt] (0.74795,-0.37947775)-- (0.74795,-0.29);
\draw [line width=0.8pt] (0.74795,-0.29)-- (0.8374277500000002,-0.29);
\draw [line width=0.8pt] (0.8374277500000002,-0.29)-- (0.8374277499999997,-0.37947774999999967);
\draw [line width=0.8pt] (0.8374277499999995,-0.4525722500000007)-- (0.74795,-0.45257225);
\draw [line width=0.8pt] (0.74795,-0.45257225)-- (0.74795,-0.54205);
\draw [line width=0.8pt] (0.74795,-0.54205)-- (0.8374277500000002,-0.54205);
\draw [line width=0.8pt] (0.8374277500000002,-0.54205)-- (0.8374277499999995,-0.4525722500000007);
\draw [line width=0.8pt] (0.8374277500000002,-0.74795)-- (0.74795,-0.74795);
\draw [line width=0.8pt] (0.74795,-0.74795)-- (0.74795,-0.8374277500000001);
\draw [line width=0.8pt] (0.74795,-0.8374277500000001)-- (0.8374277499999994,-0.8374277499999996);
\draw [line width=0.8pt] (0.8374277499999994,-0.8374277499999996)-- (0.8374277500000002,-0.74795);
\draw [line width=0.8pt] (0.74795,-0.9105222500000002)-- (0.74795,-1.);
\draw [line width=0.8pt] (0.74795,-1.)-- (0.83742775,-1.);
\draw [line width=0.8pt] (0.83742775,-1.)-- (0.8374277499999996,-0.9105222500000005);
\draw [line width=0.8pt] (0.8374277499999996,-0.9105222500000005)-- (0.74795,-0.9105222500000002);
\draw [line width=0.8pt] (0.29,-1.)-- (0.37947774999999995,-1.);
\draw [line width=0.8pt] (0.37947774999999995,-1.)-- (0.3794777499999991,-0.9105222500000005);
\draw [line width=0.8pt] (0.3794777499999991,-0.9105222500000005)-- (0.29,-0.9105222500000001);
\draw [line width=0.8pt] (0.29,-0.9105222500000001)-- (0.29,-1.);
\draw [line width=0.8pt] (0.45257225,-1.)-- (0.54205,-1.);
\draw [line width=0.8pt] (0.54205,-1.)-- (0.54205,-0.9105222500000003);
\draw [line width=0.8pt] (0.54205,-0.9105222500000003)-- (0.4525722500000005,-0.9105222500000005);
\draw [line width=0.8pt] (0.4525722500000005,-0.9105222500000005)-- (0.45257225,-1.);
\draw [line width=0.8pt] (0.29,-0.8374277499999999)-- (0.37947774999999945,-0.8374277499999996);
\draw [line width=0.8pt] (0.37947774999999945,-0.8374277499999996)-- (0.37947774999999984,-0.74795);
\draw [line width=0.8pt] (0.37947774999999984,-0.74795)-- (0.29,-0.74795);
\draw [line width=0.8pt] (0.29,-0.74795)-- (0.29,-0.8374277499999999);
\draw [line width=0.8pt] (0.4525722500000004,-0.8374277499999994)-- (0.54205,-0.8374277500000002);
\draw [line width=0.8pt] (0.54205,-0.8374277500000002)-- (0.54205,-0.74795);
\draw [line width=0.8pt] (0.54205,-0.74795)-- (0.4525722499999999,-0.74795);
\draw [line width=0.8pt] (0.4525722499999999,-0.74795)-- (0.4525722500000004,-0.8374277499999994);
\draw [line width=0.8pt] (0.29,-0.54205)-- (0.3794777499999999,-0.54205);
\draw [line width=0.8pt] (0.3794777499999999,-0.54205)-- (0.37947774999999956,-0.45257225000000045);
\draw [line width=0.8pt] (0.37947774999999956,-0.45257225000000045)-- (0.29,-0.45257225);
\draw [line width=0.8pt] (0.29,-0.45257225)-- (0.29,-0.54205);
\draw [line width=0.8pt] (0.45257224999999995,-0.54205)-- (0.54205,-0.54205);
\draw [line width=0.8pt] (0.54205,-0.54205)-- (0.54205,-0.4525722499999998);
\draw [line width=0.8pt] (0.54205,-0.4525722499999998)-- (0.4525722500000005,-0.4525722500000007);
\draw [line width=0.8pt] (0.4525722500000005,-0.4525722500000007)-- (0.45257224999999995,-0.54205);
\draw [line width=0.8pt] (0.4525722500000006,-0.3794777499999995)-- (0.54205,-0.3794777499999997);
\draw [line width=0.8pt] (0.54205,-0.3794777499999997)-- (0.54205,-0.29);
\draw [line width=0.8pt] (0.54205,-0.29)-- (0.4525722500000002,-0.29);
\draw [line width=0.8pt] (0.4525722500000002,-0.29)-- (0.4525722500000006,-0.3794777499999995);
\draw [line width=0.8pt] (0.29,-0.37947774999999995)-- (0.3794777499999992,-0.3794777499999994);
\draw [line width=0.8pt] (0.3794777499999992,-0.3794777499999994)-- (0.3794777500000001,-0.29);
\draw [line width=0.8pt] (0.3794777500000001,-0.29)-- (0.29,-0.29);
\draw [line width=0.8pt] (0.29,-0.29)-- (0.29,-0.37947774999999995);
\draw [line width=0.8pt] (-0.91052225,1.)-- (-1.,1.);
\draw [line width=0.8pt] (-1.,1.)-- (-1.,0.9105222500000001);
\draw [line width=0.8pt] (-1.,0.9105222500000001)-- (-0.9105222500000005,0.910522250000001);
\draw [line width=0.8pt] (-0.9105222500000005,0.910522250000001)-- (-0.91052225,1.);
\draw [line width=0.8pt] (-0.9105222500000004,0.8374277499999996)-- (-1.,0.83742775);
\draw [line width=0.8pt] (-1.,0.83742775)-- (-1.,0.74795);
\draw [line width=0.8pt] (-1.,0.74795)-- (-0.9105222500000002,0.74795);
\draw [line width=0.8pt] (-0.9105222500000002,0.74795)-- (-0.9105222500000004,0.8374277499999996);
\draw [line width=0.8pt] (-0.9105222500000002,0.54205)-- (-1.,0.54205);
\draw [line width=0.8pt] (-1.,0.54205)-- (-1.,0.45257225000000045);
\draw [line width=0.8pt] (-1.,0.45257225000000045)-- (-0.9105222500000006,0.4525722500000009);
\draw [line width=0.8pt] (-0.9105222500000006,0.4525722500000009)-- (-0.9105222500000002,0.54205);
\draw [line width=0.8pt] (-0.9105222500000008,0.3794777499999995)-- (-1.,0.37947775000000045);
\draw [line width=0.8pt] (-1.,0.37947775000000045)-- (-1.,0.29);
\draw [line width=0.8pt] (-1.,0.29)-- (-0.9105222500000003,0.29);
\draw [line width=0.8pt] (-0.9105222500000003,0.29)-- (-0.9105222500000008,0.3794777499999995);
\draw [line width=0.8pt] (-0.8374277499999997,0.3794777499999998)-- (-0.74795,0.3794777500000001);
\draw [line width=0.8pt] (-0.74795,0.3794777500000001)-- (-0.74795,0.29);
\draw [line width=0.8pt] (-0.74795,0.29)-- (-0.8374277500000002,0.29);
\draw [line width=0.8pt] (-0.8374277500000002,0.29)-- (-0.8374277499999997,0.3794777499999998);
\draw [line width=0.8pt] (-0.8374277499999995,0.4525722500000008)-- (-0.74795,0.4525722500000001);
\draw [line width=0.8pt] (-0.74795,0.4525722500000001)-- (-0.74795,0.54205);
\draw [line width=0.8pt] (-0.74795,0.54205)-- (-0.8374277500000001,0.54205);
\draw [line width=0.8pt] (-0.8374277500000001,0.54205)-- (-0.8374277499999995,0.4525722500000008);
\draw [line width=0.8pt] (-0.8374277500000001,0.74795)-- (-0.74795,0.74795);
\draw [line width=0.8pt] (-0.74795,0.74795)-- (-0.74795,0.8374277500000002);
\draw [line width=0.8pt] (-0.74795,0.8374277500000002)-- (-0.8374277499999992,0.8374277499999997);
\draw [line width=0.8pt] (-0.8374277499999992,0.8374277499999997)-- (-0.8374277500000001,0.74795);
\draw [line width=0.8pt] (-0.74795,0.9105222500000003)-- (-0.74795,1.);
\draw [line width=0.8pt] (-0.74795,1.)-- (-0.8374277499999999,1.);
\draw [line width=0.8pt] (-0.8374277499999999,1.)-- (-0.8374277499999995,0.9105222500000006);
\draw [line width=0.8pt] (-0.8374277499999995,0.9105222500000006)-- (-0.74795,0.9105222500000003);
\draw [line width=0.8pt] (-0.29,1.)-- (-0.37947774999999984,1.);
\draw [line width=0.8pt] (-0.37947774999999984,1.)-- (-0.379477749999999,0.9105222500000005);
\draw [line width=0.8pt] (-0.379477749999999,0.9105222500000005)-- (-0.29,0.9105222500000001);
\draw [line width=0.8pt] (-0.29,0.9105222500000001)-- (-0.29,1.);
\draw [line width=0.8pt] (-0.4525722499999999,1.)-- (-0.54205,1.);
\draw [line width=0.8pt] (-0.54205,1.)-- (-0.54205,0.9105222500000004);
\draw [line width=0.8pt] (-0.54205,0.9105222500000004)-- (-0.4525722500000004,0.9105222500000005);
\draw [line width=0.8pt] (-0.4525722500000004,0.9105222500000005)-- (-0.4525722499999999,1.);
\draw [line width=0.8pt] (-0.29,0.8374277499999999)-- (-0.37947774999999934,0.8374277499999996);
\draw [line width=0.8pt] (-0.37947774999999934,0.8374277499999996)-- (-0.3794777499999997,0.74795);
\draw [line width=0.8pt] (-0.3794777499999997,0.74795)-- (-0.29,0.74795);
\draw [line width=0.8pt] (-0.29,0.74795)-- (-0.29,0.8374277499999999);
\draw [line width=0.8pt] (-0.4525722500000003,0.8374277499999994)-- (-0.54205,0.8374277500000004);
\draw [line width=0.8pt] (-0.54205,0.8374277500000004)-- (-0.54205,0.74795);
\draw [line width=0.8pt] (-0.54205,0.74795)-- (-0.4525722499999998,0.74795);
\draw [line width=0.8pt] (-0.4525722499999998,0.74795)-- (-0.4525722500000003,0.8374277499999994);
\draw [line width=0.8pt] (-0.29,0.54205)-- (-0.37947774999999984,0.54205);
\draw [line width=0.8pt] (-0.37947774999999984,0.54205)-- (-0.3794777499999995,0.4525722500000005);
\draw [line width=0.8pt] (-0.3794777499999995,0.4525722500000005)-- (-0.29,0.45257225000000006);
\draw [line width=0.8pt] (-0.29,0.45257225000000006)-- (-0.29,0.54205);
\draw [line width=0.8pt] (-0.4525722499999999,0.54205)-- (-0.54205,0.54205);
\draw [line width=0.8pt] (-0.54205,0.54205)-- (-0.54205,0.45257224999999984);
\draw [line width=0.8pt] (-0.54205,0.45257224999999984)-- (-0.45257225000000045,0.45257225000000073);
\draw [line width=0.8pt] (-0.45257225000000045,0.45257225000000073)-- (-0.4525722499999999,0.54205);
\draw [line width=0.8pt] (-0.45257225000000056,0.37947774999999956)-- (-0.54205,0.3794777499999998);
\draw [line width=0.8pt] (-0.54205,0.3794777499999998)-- (-0.54205,0.29);
\draw [line width=0.8pt] (-0.54205,0.29)-- (-0.4525722500000001,0.29);
\draw [line width=0.8pt] (-0.4525722500000001,0.29)-- (-0.45257225000000056,0.37947774999999956);
\draw [line width=0.8pt] (-0.29,0.37947775)-- (-0.37947774999999917,0.37947774999999945);
\draw [line width=0.8pt] (-0.37947774999999917,0.37947774999999945)-- (-0.37947775000000006,0.29);
\draw [line width=0.8pt] (-0.37947775000000006,0.29)-- (-0.29,0.29);
\draw [line width=0.8pt] (-0.29,0.29)-- (-0.29,0.37947775);
\draw [line width=0.8pt] (-0.91052225,-1.)-- (-1.,-1.);
\draw [line width=0.8pt] (-1.,-1.)-- (-1.,-0.9105222500000001);
\draw [line width=0.8pt] (-1.,-0.9105222500000001)-- (-0.9105222500000005,-0.910522250000001);
\draw [line width=0.8pt] (-0.9105222500000005,-0.910522250000001)-- (-0.91052225,-1.);
\draw [line width=0.8pt] (-0.9105222500000004,-0.8374277499999996)-- (-1.,-0.83742775);
\draw [line width=0.8pt] (-1.,-0.83742775)-- (-1.,-0.74795);
\draw [line width=0.8pt] (-1.,-0.74795)-- (-0.9105222500000002,-0.74795);
\draw [line width=0.8pt] (-0.9105222500000002,-0.74795)-- (-0.9105222500000004,-0.8374277499999996);
\draw [line width=0.8pt] (-0.9105222500000002,-0.54205)-- (-1.,-0.54205);
\draw [line width=0.8pt] (-1.,-0.54205)-- (-1.,-0.45257225000000045);
\draw [line width=0.8pt] (-1.,-0.45257225000000045)-- (-0.9105222500000006,-0.4525722500000009);
\draw [line width=0.8pt] (-0.9105222500000006,-0.4525722500000009)-- (-0.9105222500000002,-0.54205);
\draw [line width=0.8pt] (-0.9105222500000008,-0.3794777499999995)-- (-1.,-0.37947775000000045);
\draw [line width=0.8pt] (-1.,-0.37947775000000045)-- (-1.,-0.29);
\draw [line width=0.8pt] (-1.,-0.29)-- (-0.9105222500000003,-0.29);
\draw [line width=0.8pt] (-0.9105222500000003,-0.29)-- (-0.9105222500000008,-0.3794777499999995);
\draw [line width=0.8pt] (-0.8374277499999997,-0.3794777499999998)-- (-0.74795,-0.3794777500000001);
\draw [line width=0.8pt] (-0.74795,-0.3794777500000001)-- (-0.74795,-0.29);
\draw [line width=0.8pt] (-0.74795,-0.29)-- (-0.8374277500000002,-0.29);
\draw [line width=0.8pt] (-0.8374277500000002,-0.29)-- (-0.8374277499999997,-0.3794777499999998);
\draw [line width=0.8pt] (-0.8374277499999995,-0.4525722500000008)-- (-0.74795,-0.4525722500000001);
\draw [line width=0.8pt] (-0.74795,-0.4525722500000001)-- (-0.74795,-0.54205);
\draw [line width=0.8pt] (-0.74795,-0.54205)-- (-0.8374277500000001,-0.54205);
\draw [line width=0.8pt] (-0.8374277500000001,-0.54205)-- (-0.8374277499999995,-0.4525722500000008);
\draw [line width=0.8pt] (-0.8374277500000001,-0.74795)-- (-0.74795,-0.74795);
\draw [line width=0.8pt] (-0.74795,-0.74795)-- (-0.74795,-0.8374277500000002);
\draw [line width=0.8pt] (-0.74795,-0.8374277500000002)-- (-0.8374277499999992,-0.8374277499999997);
\draw [line width=0.8pt] (-0.8374277499999992,-0.8374277499999997)-- (-0.8374277500000001,-0.74795);
\draw [line width=0.8pt] (-0.74795,-0.9105222500000003)-- (-0.74795,-1.);
\draw [line width=0.8pt] (-0.74795,-1.)-- (-0.8374277499999999,-1.);
\draw [line width=0.8pt] (-0.8374277499999999,-1.)-- (-0.8374277499999995,-0.9105222500000006);
\draw [line width=0.8pt] (-0.8374277499999995,-0.9105222500000006)-- (-0.74795,-0.9105222500000003);
\draw [line width=0.8pt] (-0.29,-1.)-- (-0.37947774999999984,-1.);
\draw [line width=0.8pt] (-0.37947774999999984,-1.)-- (-0.379477749999999,-0.9105222500000005);
\draw [line width=0.8pt] (-0.379477749999999,-0.9105222500000005)-- (-0.29,-0.9105222500000001);
\draw [line width=0.8pt] (-0.29,-0.9105222500000001)-- (-0.29,-1.);
\draw [line width=0.8pt] (-0.4525722499999999,-1.)-- (-0.54205,-1.);
\draw [line width=0.8pt] (-0.54205,-1.)-- (-0.54205,-0.9105222500000004);
\draw [line width=0.8pt] (-0.54205,-0.9105222500000004)-- (-0.4525722500000004,-0.9105222500000005);
\draw [line width=0.8pt] (-0.4525722500000004,-0.9105222500000005)-- (-0.4525722499999999,-1.);
\draw [line width=0.8pt] (-0.29,-0.8374277499999999)-- (-0.37947774999999934,-0.8374277499999996);
\draw [line width=0.8pt] (-0.37947774999999934,-0.8374277499999996)-- (-0.3794777499999997,-0.74795);
\draw [line width=0.8pt] (-0.3794777499999997,-0.74795)-- (-0.29,-0.74795);
\draw [line width=0.8pt] (-0.29,-0.74795)-- (-0.29,-0.8374277499999999);
\draw [line width=0.8pt] (-0.4525722500000003,-0.8374277499999994)-- (-0.54205,-0.8374277500000004);
\draw [line width=0.8pt] (-0.54205,-0.8374277500000004)-- (-0.54205,-0.74795);
\draw [line width=0.8pt] (-0.54205,-0.74795)-- (-0.4525722499999998,-0.74795);
\draw [line width=0.8pt] (-0.4525722499999998,-0.74795)-- (-0.4525722500000003,-0.8374277499999994);
\draw [line width=0.8pt] (-0.29,-0.54205)-- (-0.37947774999999984,-0.54205);
\draw [line width=0.8pt] (-0.37947774999999984,-0.54205)-- (-0.3794777499999995,-0.4525722500000005);
\draw [line width=0.8pt] (-0.3794777499999995,-0.4525722500000005)-- (-0.29,-0.45257225000000006);
\draw [line width=0.8pt] (-0.29,-0.45257225000000006)-- (-0.29,-0.54205);
\draw [line width=0.8pt] (-0.4525722499999999,-0.54205)-- (-0.54205,-0.54205);
\draw [line width=0.8pt] (-0.54205,-0.54205)-- (-0.54205,-0.45257224999999984);
\draw [line width=0.8pt] (-0.54205,-0.45257224999999984)-- (-0.45257225000000045,-0.45257225000000073);
\draw [line width=0.8pt] (-0.45257225000000045,-0.45257225000000073)-- (-0.4525722499999999,-0.54205);
\draw [line width=0.8pt] (-0.45257225000000056,-0.37947774999999956)-- (-0.54205,-0.3794777499999998);
\draw [line width=0.8pt] (-0.54205,-0.3794777499999998)-- (-0.54205,-0.29);
\draw [line width=0.8pt] (-0.54205,-0.29)-- (-0.4525722500000001,-0.29);
\draw [line width=0.8pt] (-0.4525722500000001,-0.29)-- (-0.45257225000000056,-0.37947774999999956);
\draw [line width=0.8pt] (-0.29,-0.37947775)-- (-0.37947774999999917,-0.37947774999999945);
\draw [line width=0.8pt] (-0.37947774999999917,-0.37947774999999945)-- (-0.37947775000000006,-0.29);
\draw [line width=0.8pt] (-0.37947775000000006,-0.29)-- (-0.29,-0.29);
\draw [line width=0.8pt] (-0.29,-0.29)-- (-0.29,-0.37947775);
\draw [line width=2.pt] (0.,0.) circle (1.4cm);
\end{tikzpicture}
\begin{tikzpicture}[line cap=round,line join=round,>=triangle 45,x=1.0cm,y=1.0cm]
\clip(-0.1,-0.8456198347107411) rectangle (3.05,2.9229752066115715);
\draw [line width=2.pt] (0.,2.)-- (2.,2.);
\draw [line width=2.pt] (2.,0.)-- (0.,0.);
\draw [line width=2.pt] (0.,0.)-- (0.,2.);
\draw [line width=2.pt] (0.,1.)-- (2.,1.);
\draw[line width=2.pt,smooth,samples=100,domain=2:3] plot(\x,{((\x)-3)^(4)+1});
\draw[line width=2.pt,smooth,samples=100,domain=2:3] plot(\x,{0-(((\x)-3)^(4)+1)+2});
\end{tikzpicture}
\begin{tikzpicture}[line cap=round,line join=round,>=triangle 45,x=1.2cm,y=1.2cm]
\clip(-0.05,-1.55) rectangle (1.05,1.1);
\draw[line width=2.pt,smooth,samples=100,domain=0:1] plot(\x,{(\x)^(2)});
\draw[line width=2.pt,smooth,samples=100,domain=0:1] plot(\x,{0-(\x)^(2)});
\draw [line width=2.pt] (1.,-1) -- (1.,1);
\end{tikzpicture}
\caption{From the left to the right: $1$-John domain, $1$-John domain from \cite{HarK15}, $4$-John domain, $2$-spire.}\label{fig:John}
\end{figure}
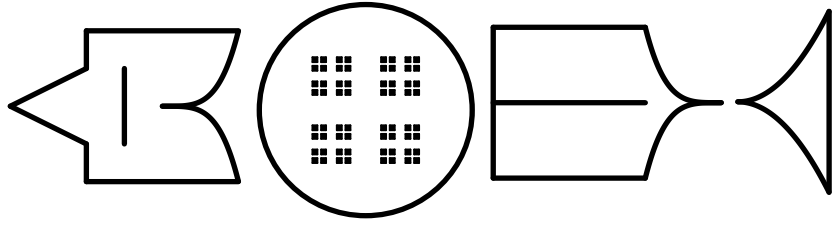

It is well known that there exists a pointwise estimate for continuously differentiable functions, that is  $C^1(\Omega)$-functions defined on a John domain,
for example
\cite[(3.4)]{HH-S_JFA} and references therein.
Less known is that there is also  a corresponding pointwise estimate for $s$-John domains.
Theorem \cite[Theorem 4.4]{HH-S_Manuscripta} with $\phi (t)=t^s$ implies the following theorem.

\begin{theorem}\label{pointwise_sJohn}
Let $s\in [1,n/(n-1) )$.
Suppose that $\Omega$ is a  bounded  $s$-John  domain in $\Rn$ with a John centre $x_0$. If 
$u\in\C^1(\Omega )$,
then for every $x\in\Omega$
\begin{equation}\label{sJohn_point-wise}
|u(x)-u_B|\le c\int_{\Omega}\frac{|\nabla u(y)|}{|x-y|^{s(n-1)}}\,dy\,.
\end{equation}
The constant $c$ does not depend on $u$, it depends only on $\Omega$.
Here $B=B(x_0,k\dist (x_0,\partial \Omega ))$ with some constant $k$  depending only on $\Omega$.
\end{theorem}

\begin{remark}
For the proofs of theorems in this section it is  crucial that the pointwise estimate \eqref{sJohn_point-wise}
holds for all $x\in\Omega$.
\end{remark}

We prove the following theorem which generalises some earlier results to the Choquet integrals with respect to the $\delta$-dimensional Hausdorff content
whenever $\Omega $ is an $s$-John domain such that $1\le  s<n/(n-1)$.

\begin{theorem}\label{TheoremPoincare}
 Let $\Omega$ be a bounded   $s$-John domain in $\R^n$, $n\geq 2,$ with
 $1\le s < \frac{n}{n-1}$. 
If $\delta\in (0,n]$ and
$p>\delta /n$, then
there is a constant $c$  which depends only on $n$, $\delta$, $p$, and  $\Omega$  such that the inequality
\begin{equation*}
\inf_{b \in \R} \int_\Omega |u(x)-b|^p \, d \Ha^{\delta}_\infty
\le c  \int_\Omega |\nabla u(x)|^p \, d \Ha^{\delta}_\infty
\end{equation*}
holds for all $u \in C^1(\Omega)$.
\end{theorem}

\begin{proof}
Choosing $\phi (t)=t^s$ in \cite[Lemma 3.1]{HH-S_Manuscripta} implies that
the assumption  (3.2) in \cite[Lemma 3.1]{HH-S_Manuscripta} is valid for the function $h(t)=t^{n+s(1-n)}$ whenever
$1\le  s< \frac{n}{n-1}$.
Thus \cite[Lemma 3.1]{HH-S_Manuscripta} gives that
\begin{equation*}
\int_{B(x,r)} \frac{|f(y)|}{|x-y|^{s(n-1)}} \, dy
\le c(n)r^{n+s(1-n)}Mf(x)
\end{equation*}
for all  $x \in \Omega$, and all $f \in L^1_{\loc}(\R^n)$.
On the other hand, by Theorem~\ref{pointwise_sJohn} for every $x\in\Omega$
\[
 |u(x) -u_B|
\le c(n,p,\Omega)
\int_{\Omega} \frac{|\nabla u(y)|}{|x-y|^{s(n-1)}} \, dy\,,
\]
where $B=B(x_0,k\dist (x_0, \partial \Omega ))$ with some constant $k\in (0,1)$ depending only on $\Omega$.
Hence, we have the pointwise estimate
\[
|u(x) -u_B|
\le  c(n,p,\Omega) \diam(\Omega)^{n+s(1-n)} M |\nabla u|(x)
\]
and  obtain 
\[
\int_\Omega |u(x) -u_B|^p d \Ha^{\delta}_\infty
\le  c(n, p, \Omega) \diam(\Omega)^{p(n+s(1-n))} \int_\Omega (M |\nabla u|(x) )^p d \Ha^{\delta}_\infty.
\]
The boundedness of the maximal operator \cite[Theorem]{OV} implies the inequality
\[
\begin{split}
\int_\Omega |u(x) -u_B|^p d \Ha^{\delta}_\infty
& \le c(n, p, \Omega, \delta) \diam(\Omega)^{p(n+s(1-n))} \int_\Omega |\nabla u(x)|^p d \Ha^{\delta}_\infty. \qedhere
\end{split}
\]

\end{proof}

\begin{remark} 
We recall that if  $1\le p <\infty$ and 
$\Omega$ in $\Rn$  is an $s$-John domain and $s<\frac{n+p-1}{n-1}$, then
the classical $(p,p)$-Poincar\'e inequality is valid for functions defined on $\Omega$.
In particular, if $1\le s <n/(n-1)$, then  the $(1,1)$-Poincar\'e inequality holds in $\Omega$ and
thus also the $(p,p)$-Poincar\'e inequality holds in $\Omega $ for all $p\geq 1$.
\end{remark}

In order to prove a Poincar\'e-Sobolev  inequality the following lemma is needed.
It is a generalisation of \cite[Lemma~3.6]{HH-S_JFA} where the case  $p\in (\delta /n, \delta )$  with $s=1$ was studied.
Now we  improve 
\cite[Lemma~3.6]{HH-S_JFA}
by  extending  the outcome  to the case $p=\frac{\delta}{n}$, too, and  also  letting $s \in [1, n/(n-1))$.

\begin{lemma}\label{lem:Choquet-Hedberg}
Let  $\delta \in (0, n]$, $s \in [1, n/(n-1))$, $p \in [\delta/n, \delta/(n+s(1-n)))$,  and $\kappa \in [0,  n + s(1-n))$. Then there exists a constant $c$ depending only on $n$, $\delta$, $\kappa$, $s$, and $p$ such that
\[
\int_\Rn \frac{|f(y)|}{|x-y|^{s(n-1)}} \, dy
\le c (M_\kappa f(x))^{1-\frac{p(n-\kappa+s(1-n))}{\delta- \kappa p}}
\bigg(\int_{\Rn} |f(y)|^{p} \, d \Ha^{\delta}_\infty \bigg)^{\frac{n-\kappa +s(1-n)}{\delta -\kappa p}} 
\]
for all $x \in \Rn$  and all $f \in L^1_{\loc}(\Rn)$.
\end{lemma}
\begin{proof}
Let us write  $A_j =\{y \in \Rn : 2^{-j} r \le |x-y|< 2^{-j+1}r\}$, $r>0$. By estimating and using the sum of the geometric series we obtain that
\[
\begin{split}
\int_{B(x, r)} \frac{|f(y)|}{|x-y|^{s(n-1)}} \, dy 
&= \sum_{j=1}^\infty \int_{A_j} \frac{|f(y)|}{|x-y|^{s(n-1)}} \, dy\\
&\le  \sum_{j=1}^\infty (2^{-j}r)^{s(1-n)} \int_{B(x, 2^{-j+1}r)} |f(y)| \, dy\\
&\le   c(n, s, \kappa)  r^{n-\kappa + s(1-n)} M_\kappa f(x)\,.
\end{split}
\]
Now we consider the case $p>\delta /n$ at first.
For the integral over  the complement of  the ball  $B(x,r)$  with respect to $\Rn$ 
Hölder's inequality and Lemma~\ref{GeneralizationOV} imply that
\[
\begin{split}
&\int_{\Rn\setminus B(x, r)} \frac{|f(y)|}{|x-y|^{s(n-1)}} \, dy\\
&\le \Big(\int_{\Rn\setminus B(x, r)} |f(y)|^{\frac{np}{\delta}}\, dy \Big)^{\frac{\delta}{np}} 
\Big(\int_{\Rn\setminus B(x, r)} |x-y|^{\frac{nps(1-n)}{np-\delta}} \, dy\Big)^{\frac{np-\delta}{np}}\\
&\le  c(n, \delta, p)\Big(\int_{\Rn\setminus B(x, r)} |f(y)|^{p} \, d \Ha^{\delta}_\infty \Big)^{\frac{1}{p}} 
\Big(\int_{\Rn\setminus B(x, r)} |x-y|^{\frac{nps(1-n)}{np-\delta}} \, dy\Big)^{\frac{np-\delta}{np}}.
\end{split}
\]
By the well-known result 
\cite[Lemma]{Hed72}
the last integral  on the right-hand side  is
\[
\int_{\Rn\setminus B(x, r)} |x-y|^{\frac{nps(1-n)}{np-\delta}} \, dy
= \frac{\omega_{n-1}}{\frac{nps(n-1)}{np-\delta} -n } r^{n-\frac{nps(n-1)}{np-\delta}},
\]
where $\omega_{n-1}$ is the $n-1$-dimensional Hausdorff measure of the sphere.
Since  $p < \frac{\delta}{n+s(1-n)}$,  the term $n-\frac{nps(n-1)}{np-\delta}$ is negative.
Thus we have
\[
\int_\Rn \frac{|f(y)|}{|x-y|^{s(n-1)}} \, dy \le c  \big( r^{n-\kappa+s(1-n)} M_\kappa f(x) +  \|f\| r^{n-\frac{\delta}{p} +s(1-n)} \big),
\]
where $\|f\| := \Big(\int_{\Rn\setminus B(x, r)} |f(y)|^{p} \, d \Ha^{\delta}_\infty \Big)^{\frac{1}{p}}$.
Choosing 
\[
r= \Big( \frac{M_\kappa f(x)}{\|f\|} \Big)^{-\frac{p}{\delta - \kappa p}}
\]
yields that
\[
\int_\Rn \frac{|f(y)|}{|x-y|^{s(n-1)}} \, dy \le c (M_\kappa f(x))^{1-\frac{p(n-\kappa+s(1-n))}{\delta- \kappa p}} \|f\|^{\frac{p(n-\kappa +s(1-n))}{\delta-\kappa p}}
\]
 for all $x \in \Rn$.
 This inequality gives the claim whenever  $p \in (\delta/n, \delta/(n+s(1-n)))$.
 
Now we consider the case $p=\frac{\delta}{n}$. 
For the integral over the complement of  the ball  $B(x,r)$  with respect to $\Rn$,  by Lemma~\ref{GeneralizationOV} we obtain that
\[
\begin{split}
\int_{\Rn\setminus B(x, r)} \frac{|f(y)|}{|x-y|^{s(n-1)}} \, dy
&\le  r^{s(1-n)} \int_{\Rn\setminus B(x, r)} |f(y)|\, dy   \\
&\le  c(n, \delta) r^{s(1-n)}\Big(\int_{\Rn\setminus B(x, r)} |f(y)|^{\frac{\delta}{n}} \, d \Ha^{\delta}_\infty\Big)^{\frac{n}{\delta}}.
\end{split}
\]
Hence,
\[
\int_\Rn \frac{|f(y)|}{|x-y|^{s(n-1)}} \, dy \le c  \big( r^{n-\kappa+s(1-n)} M_\kappa f(x) +  \|f\| r^{s(1-n)} \big),
\]
where $\|f\| = \Big(\int_{\Rn\setminus B(x, r)} |f(y)|^{\frac{\delta}{n}} \, d \Ha^{\delta}_\infty \Big)^{\frac{n}{\delta}}$ as before.
Choosing 
\[
r= \Big( \frac{M_\kappa f(x)}{\|f\|} \Big)^{-\frac{1}{n - \kappa}}
\]
gives
\[
\int_\Rn \frac{|f(y)|}{|x-y|^{s(n-1)}} \, dy \le c (M_\kappa f(x))^{\frac{s(n-1)}{n- \kappa}} \|f\|^{1-\frac{s(n-1)}{n-\kappa}}
\]
 for all $x \in \Rn$, which is  the claim  whenever $p=\delta /n$.
\end{proof}

Combining Theorem \ref{pointwise_sJohn}, Lemma \ref{lem:Choquet-Hedberg}, and the boundedness of the 
Hardy -Littlewood maximal operator with respect to the Hausdorff content
\cite[Theorem 7(a)]{Adams1998}
 implies the  following Poincar\'e-Sobolev  inequality.

\begin{theorem}\label{TheoremPoincareSobolev}
Let $\Omega$ be a bounded  $s$-John domain in $\Rn$, $1\le s <n/(n-1)$,
and let  $\delta \in (0, n]$, $p \in (\delta/n, \delta/(n+s(1-n)))$, and $\kappa \in [0,  n + s(1-n))$.
Then there exists a constant $c$ depending only on $n$, $\delta$, $\kappa$,  $p$, and $\Omega$ such that  
\begin{equation*}
\inf_{b \in \R}\Big(\int_\Omega |u(x) -b|^{q} d \Ha^{\delta- \kappa p}_\infty \Big)^{\frac{1}{q}}
\le c    \Big(\int_\Omega |\nabla u(x)|^p d \Ha^{\delta}_\infty\Big)^{\frac{1}{p}}
\end{equation*}
for all $u \in C^1(\Omega)$.
Here $q=\frac{(\delta-\kappa p)p}{\delta -p(n+s(1-n))}$.
\end{theorem}

\begin{proof}
By  combining Theorem \ref{pointwise_sJohn} and Lemma \ref{lem:Choquet-Hedberg} we obtain
\[
|u(x)-u_B|^{\frac{(\delta - \kappa p) p}{\delta - \kappa p - p(n-\kappa +s(1-n))}}
\le c 
\big(M_\kappa |\nabla u|(x)\big)^{p}
\Bigg(\int_{\Omega} |\nabla u(y)|^{p} \, d \Ha^{\delta}_\infty \bigg)^{\frac{p(n-\kappa +s(1-n))}{\delta -\kappa p -p(n-\kappa +s(1-n))}}  
\]
for all $x \in \Omega$. Here  we have 
defined $|\nabla u(x)|=0$,  when $x\in\Rn\backslash\Omega$.
By the assumption we have 
$p \in (\delta/n, \delta /\kappa)$. Thus Theorem~\ref{thm:fractional-maximal-function} yields 
 that
\[  
\int_\Rn \big(M_\kappa |\nabla u|(x) \big)^p \, d \Ha^{\delta-\kappa p}_\infty \le c \int_\Rn |\nabla u(x)|^p \, d \Ha^{\delta}_\infty
= c \int_\Omega |\nabla u(x)|^p \, d \Ha^{\delta}_\infty.
\]
These inequalities imply the claim.
\end{proof}

\begin{remark}
The previous theorem  with $s=1$ recovers the result in the sense of Choquet integrals of the classical  $1$-John domains, \cite[Theorem 3.7]{HH-S_JFA}.
Whenever $\delta=n$ and $1<p<n$, the well-known Sobolev inequality for $1$-John domains is  also recovered
\cite{Bojarski1988}.
However whenever  $s>1$, 
the previous theorem  with $\delta =n$ and $\kappa =0$ does not give the known sharp bound for $q$  \cite{KM}.
\end{remark}

If we let $p=\frac{\delta}{n}$,
a weak-type  Poincar\'e-Sobolev  estimate  holds.
A weak-type Poincar\'e-Sobolev inequality for $s=1$ has been proved in 
\cite[Theorem 3.6]{{HH-S_La}}.
Now we generalise it to the case  $1< s <n/(n-1)$.

\begin{theorem}\label{weaklimit}
Let $\Omega$ be a bounded  $s$-John domain in $\Rn$, $1\le s <n/(n-1)$,
and let  $\delta \in (0, n]$. 
If $u \in C^1(\Omega)$, then
for every $t>0$
\begin{equation*}
\inf_{b \in \R} \, \Ha_\infty^{\delta}\bigl(\{x\in\Omega :\vert u(x)-b\vert >t\}\bigr)
\le ct^{-\frac{\delta}{s(n-1)}} \Big(\int_{\Omega}\vert \nabla u(x)\vert^{\frac{\delta}{n}}\,d\Ha_\infty^{\delta} \Big)^{\frac{n}{s(n-1)}}\,,
\end{equation*}
where $c$ is  a constant which  depends only on $n$, $\delta$, and  $\Omega$.
\end{theorem}

\begin{proof}
Let us suppose first that $\delta \in (0, n)$.
Since $\Omega$ is an  $s$-John domain, Theorem~\ref{pointwise_sJohn} gives
for every $x\in\Omega$
\begin{equation}\label{pointwise}
|u(x) -u_B| \le c \int_\Omega \frac{|\nabla u(y)|}{|x-y|^{s(n-1)}}\,dy \,,
\end{equation}
where $B=B(x_0, k\dist(x_0, \partial \Omega))$ with some constant $k\in (0,1)$.
Combining inequality  \eqref{pointwise} with Lemma~\ref{lem:Choquet-Hedberg} implies that
\begin{equation}\label{pointwisemax}
|u(x) -u_B| \le c (M_\kappa |\nabla u|(x))^{\frac{s(n-1)}{n- \kappa}} N^{1-\frac{s(n-1)}{n-\kappa}},
\end{equation}
where we have written  $N:= \Big(\int_{\Rn\setminus B(x, r)} |\nabla u(y)|^{\frac{\delta}{n}} \, d \Ha^{\delta}_\infty \Big)^{\frac{n}{\delta}}$.
Next we choose $\kappa =0$.
By using  \eqref{pointwisemax}  and a weak-type estimate of the maximal operator by  \cite[Theorem 7 (ii)]{Adams1998}  or \cite[Theorem ii]{OV} we obtain
\[
\begin{split}
\Ha_\infty^{\delta}\Big(\big\{x\in\Omega :\vert u(x)-u_B\vert >t \big\}\Big)
&\le  \Ha_\infty^{\delta}\biggl(\Big\{x\in\Omega :c (M|\nabla u|(x))^{\frac{s(n-1)}n} N^{\frac{n+s(1-n)}{n}} >t \Big\}\biggr)\\
&=  \Ha_\infty^{\delta}\biggl(\Big\{x\in\Omega :  M \big(c N^{\frac{n+s(1-n)}{s(n-1)}}|\nabla u|\big)(x) >t^{\frac{n}{s(n-1)}}\Big\}\biggr)\\
&\le (t^{\frac{n}{s(n-1)}})^{-\frac{\delta}{n}} \int_{\Omega}\big(c N^{\frac{n+s(1-n)}{s(n-1)}}\vert \nabla u(x)\vert \big)^{\frac{\delta}n}\,d\Ha_\infty^{\delta}\\
&\le c t^{- \frac{\delta}{s(n-1)}} \Big(\int_{\Omega}\vert \nabla u(x)\vert^{\frac{\delta}n}\,d\Ha_\infty^{\delta} \Big)^{\frac{n}{s(n-1)}}\,.
\end{split}
\]
This gives the claim whenever $\delta\in (0,n)$.

If $\delta=n$, 
then the claim follows similarly as above.  The difference is that  the well-known  weak-type estimate of the Hardy-Littlewood maximal operator 
can be used and  also the fact that $\Ha_\infty^{n}(E) \approx |E|$ for all measurable  sets $E \subset \Rn$.
\end{proof}

\begin{remark}
(1) If functions are continuously differentiable, compactly supported functions, that is $C^1_0$-functions on domains, then the regularity of the boundary does not 
affect as it does for $C^1$-functions in Poincar\'e-Sobolev  inequalities. We refer to \cite[Chapter 4]{HH-S_JFA}.

(2) We point out that  the proofs of 
Theorems~\ref{TheoremPoincare}, \ref{TheoremPoincareSobolev}, and \ref{weaklimit}  
give stronger inequalities than in the statements in these theorems, respectively. Namely, 
in the proofs we estimate $\vert u(x)-u_B\vert $ where
$B:=B(x_0,k\dist (x_0,\partial \Omega ))$ 
and $u_B=\vert B\vert^{-1}\int_B u(x)\,dx$
with $k\in(0,1)$  being a constant and depending 
on $x_0$ and John constants of $\Omega$.
Hence, the following three inequalities hold for a bounded $s$-John domain in $\Rn$  with
$1\le s <n/(n-1)$,  $\delta \in (0, n]$.

\begin{itemize}
\item[(a)]
If $p>\delta /n$, then
there is a constant $c$  which depends only on $n$, $\delta$, $p$, and  $\Omega$  such that the inequality
\begin{equation*}
 \int_\Omega |u(x)-u_B|^p \, d \Ha^{\delta}_\infty
\le c  \int_\Omega |\nabla u(x)|^p \, d \Ha^{\delta}_\infty
\end{equation*}
is valid for
all $u\in C^1(\Omega)$.
\item[(b)]
Let
$p \in (\delta/n, \delta/(n+s(1-n)))$ and $\kappa \in [0,  n + s(1-n))$.
Then there exists a constant $c$ depending only on $n$, $\delta$, $\kappa$, $p$, and $\Omega$ such that  
\begin{equation*}
\Big(\int_\Omega |u(x) -u_B|^{q} d \Ha^{\delta- \kappa p}_\infty \Big)^{\frac{1}{q}}
\le c    \Big(\int_\Omega |\nabla u(x)|^p d \Ha^{\delta}_\infty\Big)^{\frac{1}{p}}
\end{equation*}
for all $u \in C^1(\Omega)$.
Here $q=\frac{(\delta-\kappa p)p}{\delta -p(n+s(1-n))}$.
\item[(c)]
If $u\in C^1(\Omega)$, then
for every $t>0$
\begin{equation*}
 \Ha_\infty^{\delta}\bigl(\{x\in\Omega :\vert u(x)-u_B\vert >t\}\bigr)
\le ct^{-\frac{\delta}{s(n-1)}} \Big(\int_{\Omega}\vert \nabla u(x)\vert^{\frac{\delta}{n}}\,d\Ha_\infty^{\delta} \Big)^{\frac{n}{s(n-1)}}\,,
\end{equation*}
where $c$ is a constant which depends only on $n$, $\delta$, and $\Omega$. 
\end{itemize}
\end{remark}

By Theorems \ref{TheoremPoincare}-\ref{weaklimit}
we are able to write corollaries for $s$-spires, that is for $s$-power cusps.
In particular, Theorem \ref{TheoremPoincare} implies Corollary \ref{PoincareCorollary}.
Theorems \ref{TheoremPoincareSobolev}
and \ref{weaklimit} give Poincar\'e-Sobolev inequalities for $s$-power cusps.
At least part $(a)$ 
in the following corollary is not sharp.  If $\delta =n $ and $\kappa =0$, for
 the sharp exponent on the left-hand side   in 
 %\eqref{cuspPoincareSobolev} 
 the inequality of the function in part (a)
 we refer to  \cite{Mazya}, \cite{MazyaPoborchi}.

\begin{corollary}
Let $\delta\in (0,n]$, 
$s\in [1,n/(n-1)$, and let
$\Omega$ be an $s$-spire in $\Rn$, $n\geq 2$, that is
$
\Omega =\{(x_1, x_2, \dots , x_n)\in (0,1)\times \R^{n-1}\, : \vert\vert (x_2, x_3, \dots ,x_n)\vert\vert < x^s\}\,.
$
\begin{itemize}
\item[(a)] \,
If $u \in C^1(\Omega)$
and
$p \in (\delta/n, \delta/(n+s(1-n)))$, then
\begin{equation*}%\label{cuspPoincareSobolev}
\inf_{b \in \R}\Big(\int_\Omega |u(x) -b|^{\frac{\delta p}{\delta - p(n+s(1-n))}} d \Ha^{\delta- \kappa p}_\infty \Big)^{\frac{\delta -p(n+s(1-n))}{  \delta p}}                             
\le c    \Big(\int_\Omega |\nabla u(x)|^p d \Ha^{\delta}_\infty\Big)^{\frac{1}{p}}\,
\end{equation*}
where  $c$ is a constant which depends only on $n$, $\delta$, $p$, and $s$.
\item[(b)] \,
If $u \in C^1(\Omega)$ and
$p=\delta /n$, then 
for every $t>0$
\begin{equation*}
\inf_{b \in \R} \, \Ha_\infty^{\delta}\bigl(\{x\in\Omega :\vert u(x)-b\vert >t\}\bigr)
\le ct^{-\frac{\delta}{s(n-1)}} \Big(\int_{\Omega}\vert \nabla u(x)\vert^{\frac{\delta}{n}}\,d\Ha_\infty^{\delta} \Big)^{\frac{n}{s(n-1)}}\,
\end{equation*}
where $c$ is 
a constant which depends only on $n$, $\delta$, $p$, and $s$.
\end{itemize}
\end{corollary}

\bibliographystyle{amsalpha}

\begin{thebibliography}{HH}

\bibitem{ADLG2013}
Acosta, G., Dur\'an, R. G.,   L\'opez Garcia,  F.:
{Korn inequality and divergence operator: Counterexamples and optimality of weighted estimates}.
Proc. Amer. Math. Soc.\text{141}, no.~1, 217--232 (2013)

\bibitem{Adams1975}
Adams, D. R.: {A note on Riesz potentials}. Duke Math.\ J. \textbf{42},  no.~4, 765--778 (1975)

\bibitem{Adams1986}
 Adams, D. R.:  {A note on  Choquet integrals with respect to Hausdorff capacity}. 
In:  Cwikel, M.,  Peetre, J.,  Sagher, Y.,  Wallin, H. (eds.) Function Spaces and Applications (Lund 1986), 
 Lecture Notes in Mathematics  vol. 1302,  Springer, Berlin (1988)  pp. 115--124

\bibitem{Adams1998}
 Adams, D. R.:
{Choquet Integrals in Potential Theory}.
Publ. Mat. \textbf{42}, {3--66}  (1998)

\bibitem{Adams2015}
Adams, D. R.: {Morrey Spaces}. {Birkhäuser, Cham--Heidelberg--New York} (2015)

%\bibitem{AdamsHedberg}
%Adams, D. R.,  Hedberg, L. I.:
%Function spaces and potential theory.
%Grundlehren der mathematischen Wissenschaften \textbf{314}
%A series of Comprehensive Studies in Mathematics
%{Springer-Verlag, Berlin--Heidelberg--New York} (1999)

\bibitem{AX2012}
Adams D.R., Xiao, J.: Morrey spaces in harmonic analysis.
Ark. Mat. \textbf{50}, 201--230 (2012)

\bibitem{Bojarski1988}
 Bojarski, B.:
{Remarks on Sobolev imbedding inequalities}.
In: {Complex Analysis, Joensuu, 1987,} 
Lecture Notes in Math.,
\textbf{1351} {Springer, Berlin} (1988) %52--68,

\bibitem{ChenOoiSpector2023}
Chen, Y-W.,  Ooi, K.H., Spector,  D.: Capacitary maximal inequalities and applications.
 \texttt{https://arxiv.org/abs/2305.19046v1}.
  
\bibitem{ChenSpector2023}
Chen, Y-W, Spector, D: 
On functions of bounded $\beta$-dimensional mean oscillation.
Adv. Calc. Var. (2023), doi:10.1515/acv-2022-0084
\bibitem{Cho53}
 Choquet, G.: Theory of capacities. {Ann. Inst. Fourier (Grenoble)}~\textbf{5},  13--295 (1953--1954)

\bibitem{DX}
 Dafni, G., Xiao, J.:
Some new tent spaces and duality theorems for fractional Carleson measures and $Q_{\alpha}(\Rn)$.
{J. Funct. Anal.}\textbf{208},  377--422 (2004)

\bibitem{Den94}
Denneberg, D.: Non-additive measure and integral Theory and Decision Library. 
Series B: Mathematical and Statistical Methods vol. 27,  Kluwer Academic Publishers Group, Dordrecht (1994)


\bibitem{EH-S}
Edmunds, D. E., Hurri-Syrj\"anen, R.:
The improved Hardy inequality.
{Houston J. Math.} \textbf{37}, 929--937 (2011)




\bibitem{Federer}
Federer, H.:
{Geometric Measure Theory}. Die Grundlehren der mathematischen Wissenschaften Band 153, 
{Springer-Verlag New York Inc.,  New York} (1969) 

\bibitem{GT}
Gilbarg, D., Trudinger, N.: Elliptic partial differential equations of second order.
Springer-Verlag,  Heidelberg (1983)

\bibitem{Hajlasz1999}
 Haj{\l}asz, P.: {Pointwise Hardy inequalities}.
Proc. Amer. Math. Soc. \textbf{127},  no.~2, 417--423 (1999)

\bibitem{HH-S_Manuscripta}
Harjulehto, P., Hurri-Syrj\"anen, R.:
{Pointwise estimates to the modified Riesz potential}. 
 Manuscripta Math. \textbf{156},  no.~3, 521--543 (2018)

\bibitem{HH-S_JFA}
Harjulehto, P., Hurri-Syrj\"anen, R.: {On Choquet integrals and Poincar\'e-Sobolev  type inequalities}.  
J.\ Funct.\ Anal.\ \textbf{284}, issue 9, paper no.\ 109862, (2023)

\bibitem{HH-S_AAG}
Harjulehto, P., Hurri-Syrj\"anen, R.: 
Estimates for the variable order Riesz potential with application. 
In: 
Lenhart, S,  Xiao, J (eds)
Potentials and Partial Differential Equations: The Legacy of David R. Adams   
Advances in Analysis and Geometry vol. 8, De Gruyter, Berlin  (2023) pp. 127--155

\bibitem{HH-S_La}
Harjulehto, P., Hurri-Syrj\"anen, R.: On Sobolev inequalities with Choquet integrals. submitted manuscript

\bibitem{HarK15}
Harjulehto, P., Kl{\'e}n, R.: Examples of fractals satisfying the quasihyperbolic boundary condition. 
Aust. J. Math. Anal. Appl. \textbf{12}, Issue 1, Article 9, pp. 1-12 (2015)

  
%\bibitem{John}
%John, F.: Rotation and strain. 
%Comm. Pure Appl. Math. \textbf{14}, 391--413 (1961)
 
\bibitem{Hed72}
 Hedberg, L. I.: On certain convolution inequalities. {Proc.\ Amer.\ Math.\ Soc.} \textbf{36},  505--510 (1972)
   
\bibitem{Kawabe19}
 Kawabe, J.:
Convergence in Measure Theorems of the Choquet Integral Revisited.
\textit{Modeling decisions for artificial intelligence, }
Lecture Notes in Comput. Sci. \textbf{11676}
Lecture Notes in Artificial Intelligence, 
{Springer Cham} (2019) pp. 17--28

\bibitem{KM}
 Kilpel\"ainen, T., Mal\'y J.:
 Sobolev inequalities on sets with irregular boundaries.
 {Z. Anal. Angew.} \textbf{19}, 369--380 (2000)


     
\bibitem{Lewis}  
 Lewis, J.:
 Uniformly fat sets.
 {Trans.\ Amer.\ Math.\ Soc.}\textbf{308}, 177--196 (1988)
  
  
\bibitem{Liu16}
Liu, L.:
Hausdorff content and the Hardy-Liitlewood maximal operator on metric measure space.
{J. Math. Anal Appl.} \textbf{443},  732--751 (2016)

\bibitem{MartinezSpector2021}
 Mart\'{i}nez, \'{A}.\ D., Spector, D.:
An improvement to the John-Nirenberg inequality for functions in critical Sobolev spaces.
{Adv. Nonlinear Anal. }\textbf{10},  877--894 (2021)

\bibitem{Mazya}
Maz'ya, V. G.:
Sobolev Spaces.
Springer-Verlag, Berlin (1985)

\bibitem{MazyaPoborchi}
Maz'ya V. G.,   Poborchi, S. V.:
Differentiable functions on bad domains.
World Scientific, Singapore (1997)

\bibitem{OoiPhuc22}
Ooi, K. H.,  Phuc, N. C.:
The Hardy-Littlewood maximal function, Choquet integrals, and embeddings of Sobolev type.
{Math. Ann.} 382, 1865--1879 (2022)

\bibitem{OV}
 Orobitg, J., Verdera, J.:
Choquet integrals, Hausdorff content and the Hardy-Littlewood maximal operator.
{Bull. London Math. Soc.} \textbf{30}, 145--150 (1998)

\bibitem{PonceSpector20}
 Ponce, A., Spector, D.:
A boxing inequality for the fractional perimeter.
{Ann. Sc. Norm. Super. Pisa Cl. Sci.}(5)
\textbf{20}, 107--141 (2020)

\bibitem{PonceSpector23}
Ponce, A.,  Spector, D.: 
Some remarks on Capacitary Integrals and Measure Theory.
In: 
Lenhart, S,  Xiao, J (eds)
Potentials and Partial Differential Equations: The Legacy of David R. Adams   
Advances in Analysis and Geometry vol. 8, De Gruyter, Berlin  (2023)  pp. 127--155

\bibitem{Saito2019}
Saito, H.:
Boundedness of the strong maximal operator with theHausdorff content.
{Bull. Korean Math. Soc.} 56, no.~2, 399--406 (2019)

\bibitem{SaitoTanaka2022}
Saito, H., Tanaka, H.:
Dual of the Choquet spaces with general Hausdorff content.
{Studia Math.} 266, no.~3, 323--335 (2022)

\bibitem{SaitoTanakaWatanabe2016}
Saito, H., Tanaka, H., Watanabe, T.:
Abstract dyadic cubes, maximal operators and Hausdorff content.
{Bull. Sci. Math.} 140, no.~6, 757--773 (2016)

\bibitem{SaitoTanakaWatanabe2019}
Saito, H., Tanaka, H., Watanabe, T.:
Fractional maximal operators with weighted Hausdorff content.
{Positivity} 23, no.~1, 125--138 (2019)


\bibitem{Tang}
Tang, L.: 
Choquet integrals, weighted Hausdorff content and maximal operators.
Georgian Math. J 18, 587--596 (2011)

\bibitem{Wannebo}
Wannebo, A.: Hardy inequalities.
{Proc.\ Amer.\ Math.\ Soc.} \textbf{109}, 85--95, (1990)



\bibitem{YangYuan08}
Yang, D., Yuan,  W.: A note on dyadic Hausdorff capacities. 
{Bull.\ Sci.\ Math.}\ \textbf{132}, 500--509 (2008)

\bibitem{Ziemer89}
 Ziemer, W. P.:  Weakly Differentiable Functions.
   Graduate Texts in Mathematics, 120, \textit{Springer-Verlag, New York,} (1989).

\end{thebibliography}

\end{document}